\documentclass[reqno,centertags]{amsart}

\usepackage{amssymb,amsmath,amsthm} \usepackage{dsfont}

\newcommand{\R}{\mathbb{R}} \newcommand{\Z}{\mathbb{Z}}
 \newcommand{\C}{\mathbb{C}}
\newcommand{\N}{\mathbb{N}}

\theoremstyle{plain} \newtheorem{theorem}{Theorem}[section]

\newtheorem{corollary}[theorem]{Corollary}
\newtheorem{prop}[theorem]{Proposition}

\theoremstyle{definition} \newtheorem{definition}[theorem]{Definition}
\theoremstyle{remark} \newtheorem{remark}{Remark}

\DeclareMathOperator{\supp}{supp}

\newcommand*{\hx}[3]{\ensuremath{\dot{X}^{#1,#2,#3}}}
\newcommand*{\hb}[3]{\ensuremath{\dot{B}^{#1}_{#2,#3}}}
\newcommand{\eps}{\varepsilon} \newcommand{\lb}{\langle}
\newcommand{\rb}{\rangle}
\newcommand{\ls}{\lesssim}\newcommand{\gs}{\gtrsim}

\begin{document}

\author[M.~Hadac]{Martin~Hadac} \email{hadac@math.uni-bonn.de}
\author[S.~Herr]{Sebastian~Herr} \email{herr@math.uni-bonn.de}
\author[H.~Koch]{Herbert~Koch} \email{koch@math.uni-bonn.de}

\address{%
  Rheinische Friedrich-Wilhelms-Universit\"at Bonn, Mathematisches Institut,
  Bering\-stra{\ss}e 1, 53115 Bonn, Germany.}

\title[The KP-II equation in a critical space]{%
  Well-posedness and scattering for the KP-II equation in a critical space}

\begin{abstract}
  The Cauchy problem for the Kadomtsev-Petviashvili-II equation $(u_t
  +u_{xxx}+uu_x)_x+u_{yy}=0$ is considered. A small data global
  well-posedness and scattering result in the scale invariant,
  non-isotropic, homogeneous Sobolev space $\dot H^{-\frac12,0}(\R^2)$
  is derived. Additionally, it is proved that for arbitrarily large
  initial data the Cauchy problem is locally well-posed in the
  homogeneous space $\dot H^{-\frac12,0}(\R^2)$ and in the
  inhomogeneous space $H^{-\frac12,0}(\R^2)$, respectively.
\end{abstract}



\maketitle

\section{Introduction and main result}\label{sect:intro_main}
\noindent
The Kadomtsev-Petviashvili-II (KP-II) equation
\begin{equation}\label{eq:kpII}
  \begin{split}
    \partial_x (\partial_t u +\partial_{x}^3 u+u \partial_x
    u)+\partial_y^2 u=0 &
    \quad \text{ in } (0,\infty) \times \R^2\\
    u(0,x,y)=u_0(x,y) & \quad (x,y)\in \R^2
  \end{split}
\end{equation}
has been introduced by B.B.~Kadomtsev and V.I.~Petviashvili
\cite{kadomtsev:70} to describe weakly transverse water waves in the
long wave regime with small surface tension. It generalizes the
Korteweg - de Vries equation, which is spatially one dimensional and
thus neglects transversal effects. The KP-II equation has a remarkably
rich structure. Let us begin with its symmetries and assume that $u$
is a solution of \eqref{eq:kpII}.
\begin{enumerate}
\item\label{it:symm1} \emph{Translation:} Translates of $u$ in $x$,
  $y$ and $t$ are solutions.
\item\label{it:scaling} \emph{Scaling:} If $\lambda > 0$ then also
  \begin{equation}\label{eq:scaling}
    u_\lambda(t,x,y) = \lambda^2 u(\lambda^3 t, \lambda x, \lambda^2 y)
  \end{equation}
  is a solution.
\item\label{it:symm3} \emph{Galilean invariance:}
  For all $c\in \R$ the function
  \begin{equation}\label{eq:galilei}
    u_c(t,x,y) = u(t,x-cy-c^2t,y+2ct)
  \end{equation}
  satisfies equation \eqref{eq:kpII}.
\end{enumerate}

The KP-II equation is integrable in the sense that there exists a Lax
pair.  Formally, there exists an infinite sequence of conserved
quantities \cite{SZ80}, the two most important beeing
\begin{equation*}
  I_0 =  \frac12 \int  u^2 dx dy
\end{equation*}
and
\begin{equation*}
  I_1 = \frac12  \int (\partial_x u)^2 - \frac13 u^3 - (\partial_x^{-1}
  \partial_y u)^2 dx dy.
\end{equation*}
The conserved quantities besides $I_0$ seem to be useless for proofs
of well-posedness, because of the difficulty to define
$\partial_x^{-1}$ and because the quadratic term is indefinite.

There are many explicit formulas for solutions, see \cite{MR1964513}.
Particular solutions are the line solitons coming from solitons of the
Korteweg - de Vries equation, their Galilei transforms, and multiple
line soliton solutions with an intricate structure, see
\cite{MR2307861}.

It may be possible to apply the machinery of inverse scattering to
solve the initial value problem and to obtain asymptotics for
solutions, see \cite{Kiselev06} for some results in that direction. It
is however not clear which classes of initial data can be treated.

The line solitons are among the simplest solutions. An analysis of the
spectrum of the linearization and inverse scattering indicate that the
line soliton is stable \cite{kadomtsev:70,RT07}.  A satisfactory
nonlinear stability result for the line soliton is an outstanding
problem.

In this paper we want to make a modest step towards this challenging
question: We prove well-posedness and scattering in a critical space.
These results are in remarkable contrast to the situation for the
Korteweg - de Vries equation where the critical space is
$H^{-\frac32}(\R)$ and iteration techniques, as employed in the
present work, are known \cite{CCT03} to fail for initial data below
$H^{-\frac34}(\R)$.  Stability of solitons has been proved by inverse
scattering techniques and by convexity arguments using conserved
quantities \cite{MR2109467} which has no chance to carry over to KP-II
because the quadratic part of $I_1$ is not convex.

We study the Cauchy problem \eqref{eq:kpII} for initial data $u_0$ in
the non-isotropic Sobolev space $H^{-\frac12,0}(\R^2)$ and in the
homogeneous variant $\dot H^{-\frac12,0}(\R^2)$, respectively, which
are defined as spaces of distributions with $-\frac12$ generalized
$x$-deriva\-tives in $L^2(\R^2)$, see \eqref{eq:inhom_sob} and
\eqref{eq:hom_sob} at the end of this section.  These spaces are
natural for KP-II equation because of the following considerations:
The homogeneous space $\dot H^{-\frac12,0}(\R^2)$ is invariant under
the scaling symmetry \eqref{eq:scaling} of solutions of the KP-II
equation as well as under the action of the Galilei transform
\eqref{eq:galilei} for fixed $t$. Any Fourier multiplier $m$ invariant
under scaling and reflection satisfies $m(\xi,\eta) = |\xi|^{-1/2}
m(1,\eta/|\xi|^2)$.  Galilean invariance now implies that $m$ is
independent of $\eta$.

While in the super-critical range, i.e.  $s<-\frac12$, the scaling
symmetry suggests ill-posedness of the Cauchy problem (cp.\ also
\cite{KenigZiesler05} Theorem 4.2), we will prove global
well-posedness and scattering in $\dot H^{-\frac12,0}(\R^2)$ for small
initial data, see Theorem~\ref{thm:gwp_hom} and Corollary~\ref{cor:scatter},
and local well-posedness in $H^{-\frac12,0}(\R^2)$
and $\dot H^{-\frac12,0}(\R^2)$ for arbitrarily large initial data,
see Theorem~\ref{thm:lwp_hom_inhom}.

After J.~Bourgain \cite{bourgain93} established global well-posedness
in $L^2(\mathbb{T}^2)$ and $L^2(\mathbb{\R}^2)$ by the Fourier
restriction norm method and opened up the way towards a low regularity
well-posedness theory, there has been a lot of progress in this line
of research. We will only mention the most recent results and also
refer to the references therein.  Local well-posedness in the full
sub-critical range $s>-\frac12$ was obtained by H.~Takaoka
\cite{takaoka01_2} in the homogeneous spaces and by the first author
\cite{Hadac07} in the inhomogeneous spaces.  Global well-posedness for
large, real valued data in $H^{s,0}(\R^2)$ has been pushed down to
$s>-\frac{1}{14}$ by P.~Isaza - J.~Mej{\'{\i}}a \cite{isaza03}.

The first main result of this paper is concerned with small data
global well-posedness in $\dot H^{-\frac12,0}(\R^2)$.  For $\delta>0$
we define
\begin{equation*}
  \dot B_\delta :=\{u_0 \in \dot H^{-\frac12,0}(\R^2) \mid \|u_0\|_{\dot
    H^{-\frac12,0}}<\delta\},
\end{equation*}
and obtain the following:
\begin{theorem}\label{thm:gwp_hom}
  There exists $\delta>0$, such that for all initial data $u_0 \in
  \dot B_\delta$ there exists a solution
  \begin{equation*}
    u\in \dot Z^{-\frac12}([0,\infty))\subset
    C([0,\infty);\dot H^{-\frac12,0}(\R^2))
  \end{equation*}
  of the KP-II equation \eqref{eq:kpII} on $(0,\infty)$. If for some
  $T>0$ a solution $v\in
  Z^{-\frac12}([0,T])$ on $(0,T)$ satisfies
  $v(0)=u(0)$, then $v=u|_{[0,T]}$.
  Moreover, the flow map
  \begin{equation*}
    F_+: \dot B_\delta \to \dot Z^{-\frac12}([0,\infty)), \,  u_0\mapsto u 
  \end{equation*}
  is analytic.
\end{theorem}
In order to state the second main result of this paper let us define
\begin{equation*}
  B_{\delta,R}:=\{u_0 \in H^{-\frac12,0}(\R^2) \mid
  u_0=v_0+w_0, \, \|v_0\|_{\dot H^{-\frac12,0}}<\delta, \|w_0\|_{L^2}<R\},
\end{equation*}
for $\delta>0, R>0$. We establish local well-posedness for arbitrarily
large initial data, both in $H^{\frac12,0}(\R^2)$ and $\dot
H^{\frac12,0}(\R^2)$:
\begin{theorem}\label{thm:lwp_hom_inhom}
  \begin{enumerate}\item\label{it:lwp_inhom}
    There exists $\delta>0$ such that for all $R\geq \delta$ and $u_0
    \in B_{\delta,R}$ there exists a solution
    \begin{equation*}
      u\in Z^{-\frac12}([0,T])\subset C([0,T];H^{-\frac12,0}(\R^2))
    \end{equation*}
    for $T:=\delta^6 R^{-6}$ of the KP-II equation \eqref{eq:kpII} on
    $(0,T)$.  If a solution
    $v\in Z^{-\frac12}([0,T])$ on $(0,T)$
    satisfies $v(0)=u(0)$, then $v=u|_{[0,T]}$.
    Moreover, the flow map
    \begin{equation*}
      B_{\delta,R} \ni u_0\mapsto u \in  Z^{-\frac12}([0,T])
    \end{equation*}
    is analytic.
  \item\label{it:lwp_hom} The statement in Part~\ref{it:lwp_inhom}
    remains valid if we replace the space $H^{-\frac12,0}(\R^2)$ by
    $\dot H^{-\frac12,0}(\R^2)$ as well as $Z^{-\frac12}([0,T])$ by
    $\dot Z^{-\frac12}([0,T])$.
  \end{enumerate}
\end{theorem}
\begin{remark}\label{rem:def_sp}
  For the definition of the spaces $\dot Z^{-\frac12}(I)$ and
  $Z^{-\frac12}(I)$ we refer the reader to Definition~\ref{def:res_space}
  and the subsequent Remark~\ref{rem:restr}.
  In particular, we have the embedding
  $\dot Z^{-\frac12}(I)\subset Z^{-\frac12}(I)$.
  Moreover, a solution of the KP-II equation
  \eqref{eq:kpII} is understood to be a solution of the corresponding
  operator equation \eqref{eq:int_eq}, compare Section~\ref{sect:proof_main}.
\end{remark}
\begin{remark}\label{rem:time_rev}
  Due to the time reversibility of the KP-II equation, the above
  Theorems also hold in corresponding intervals $(T,0)$, $-\infty\leq
  T<0$. We denote the flow map with respect to $(-\infty,0)$ by $F_-$.
\end{remark}
\begin{remark}\label{rem:toe_freq_prof} For each $u_0 \in
  H^{-\frac12,0}(\R^2)$ and $\delta>0$ there exists $N>0$ such that
  $\|P_{\geq N} u_0\|_{\dot H^{-\frac12,0}}<\delta$.  We obviously
  have the representation $u_0=P_{\geq N} u_0 + P_{< N} u_0$, thus
  $u_0\in B_{\delta,R}$ for some $R>0$.  However, the time of local
  existence provided by Theorem~\ref{thm:lwp_hom_inhom} for large data
  may depend on the profile of the Fourier transform of $u_0$, not
  only on its norm.
\end{remark}
\begin{remark}\label{rem:pers_reg} The well-posedness
  results above are presented purely at the critical level of
  regularity $s=-\frac12$ as this is the most challenging case.  As
  the reader will easily verify by the standard modification of our
  arguments, the estimates also imply persistence of higher initial
  regularity.
\end{remark}
A consequence of Theorem~\ref{thm:gwp_hom} is scattering in $\dot
H^{-\frac12,0}(\R^2)$.
\begin{corollary}\label{cor:scatter}
  Let $\delta>0$ be as in Theorem~\ref{thm:gwp_hom}.  For every $u_0
  \in \dot B_\delta$ there exists $u_\pm \in \dot
  H^{-\frac12,0}(\R^2)$ such that
  \begin{equation*}
    F_\pm (u_0)(t)-e^{tS}u_\pm \to 0 \text{ in } \dot H^{-\frac12,0}(\R^2)
    \text{ as } t\to \pm\infty,
  \end{equation*}
  The maps
  \begin{equation*}
    V_\pm: \dot B_\delta \to \dot H^{-\frac12,0}(\R^2), \, u_0\mapsto u_\pm
  \end{equation*}
  are analytic, respectively.  For $u_0\in L^2(\R^2)\cap \dot
  B_\delta$ we have
  \begin{equation*}
    \|V_\pm(u_0)\|_{L^2}=\|u_0\|_{L^2}.
  \end{equation*}
  Moreover, the local inverses, the wave operators
  \begin{equation*}
    W_\pm: \dot B_\delta \to \dot H^{-\frac12,0}(\R^2), \, u_\pm \mapsto u(0)
  \end{equation*}
  exist and are analytic, respectively. For $u_\pm\in L^2(\R^2)\cap
  \dot B_\delta$ we have
  \begin{equation*}
    \|W_\pm(u_\pm)\|_{L^2}=\|u_0\|_{L^2}.
  \end{equation*}
\end{corollary}
\subsection*{Organization of the paper}
At the end of this section we introduce some notation.  In
Section~\ref{sect:disp_est} we review function spaces related to the
well-posedness theory for nonlinear dispersive PDE's, with a focus on
the recently introduced $U^p$ space in this context due to D. Tataru
and one of the authors, cp.\ \cite{KT05,KT07} and references therein,
as well as the closely related $V^p$ space due to N. Wiener
\cite{Wiener24}.  We believe that the techniques are useful and of
independent interest. For that reason we devoted a considerable effort
to the presentation of the methods even though most of the details are
implicitly contained in \cite{KT05,KT07}. Proposition~\ref{prop:interpolation}
however seems to be new.  In Section~\ref{sect:bil_est} we prove bilinear
estimates related to the KP-II
equation. These are the main ingredients for the proofs of our main
results, which are finally presented in Section~\ref{sect:proof_main}.
\subsection*{Notation}
The non-isotropic Sobolev spaces $H^{s_1,s_2}(\R^2)$ and $\dot
H^{s_1,s_2}(\R^2)$ are spaces of complex valued temperate
distributions, defined via the norms
\begin{align}
  \|u\|_{H^{s_1,s_2}}&:=\left( \int_{\R^2} \lb \xi\rb^{2s_1}\lb
    \eta\rb^{2s_2} |\widehat{u}(\xi,\eta)|^2
    d\xi d\eta\right)^{\frac12},\label{eq:inhom_sob}\\
  \|u\|_{\dot H^{s_1,s_2}}&:= \left( \int_{\R^2}
    |\xi|^{2s_1}|\eta|^{2s_2} |\widehat{u}(\xi,\eta)|^2 d\xi d\eta
  \right)^{\frac12},\label{eq:hom_sob}
\end{align}
respectively, where $\lb\xi\rb^2=1+|\xi|^2$. The $n$-dimensional
Fourier transform is defined as
\begin{equation*}
  \widehat{u}(\mu)=\mathcal{F}u(\mu)=(2\pi)^{-\frac n2}
  \int_{\R^n} e^{-ix \cdot \mu}u(x) dx
\end{equation*}
for $u \in L^1(\R^n)$, and extended to $\mathcal{S}'(\R^n)$ by
duality.  For $1\leq p \leq \infty$ we define the dual exponent $1\leq
p'\leq \infty$ by
\begin{equation*}
  \frac1p+\frac1{p'}=1.
\end{equation*}
\section{Function spaces and dispersive
  estimates}\label{sect:disp_est}
\noindent
In this section we discuss properties of function spaces of $U^p$ and
$V^p$ type \cite{KT05,KT07,Wiener24}.  In particular, we present
embedding results and a rigorous duality statement as well as
interpolation properties and an extension lemma for dispersive
estimates.  Though many aspects of these spaces are well known, the
interpolation result of Proposition~\ref{prop:interpolation} seems to
be new.

Let $\mathcal{Z}$ be the set of finite partitions
$-\infty=t_0<t_1<\ldots<t_K=\infty$ and let $\mathcal{Z}_0$ be the set
of finite partitions $-\infty<t_0<t_1<\ldots<t_K<\infty$.  In the
following, we consider functions taking values in $L^2:=L^2(\R^d;\C)$,
but in the general part of this section $L^2$ may be replaced by an
arbitrary Hilbert space.
\begin{definition}\label{def:u}
  Let $1\leq p <\infty$. For $\{t_k\}_{k=0}^K \in \mathcal{Z}$ and
  $\{\phi_k\}_{k=0}^{K-1} \subset L^2$ with
  $\sum_{k=0}^{K-1}\|\phi_k\|_{L^2}^p=1$ and $\phi_0=0$ we call the
  function $a:\R \to L^2$ given by
  \begin{equation*}
    a=\sum_{k=1}^{K} \mathds{1}_{[t_{k-1},t_{k})}\phi_{k-1}
  \end{equation*}
  a $U^p$-atom. Furthermore, we define the atomic space
  \begin{equation*}
    U^p:=\left\{u=\sum_{j=1}^\infty \lambda_j a_j \;\Big|\; a_j \text{ $U^p$-atom},\;
      \lambda_j\in \C \text{ s.th. } \sum_{j=1}^\infty |\lambda_j|<\infty \right\}
  \end{equation*}
  with norm
  \begin{equation}\label{eq:norm_u}
    \|u\|_{U^p}:=\inf \left\{\sum_{j=1}^\infty |\lambda_j|
      \;\Big|\; u=\sum_{j=1}^\infty \lambda_j a_j,
      \,\lambda_j\in \C,\; a_j \text{ $U^p$-atom}\right\}.
  \end{equation}
\end{definition}

\begin{prop}\label{prop:u}
  Let $1\leq p< q < \infty$.
  \begin{enumerate}
  \item\label{it:u_banach} $U^p$ is a Banach space.
  \item\label{it:u_emb} The embeddings $U^p\subset U^q\subset
    L^\infty(\R;L^2)$ are continuous.
  \item\label{it:u_right_cont} For $u\in U^p$ it holds
    $\lim_{t\downarrow t_0} \|u(t)-u(t_0)\|_{L^2}=0$, i.e.\ every $u\in
    U^p$ is right-continuous.
  \item\label{it:u_limits} $u(-\infty):=\lim_{t\to - \infty}u(t)=0$,
    $u(\infty):=\lim_{t\to \infty}u(t)$ exists.
  \item\label{it:u_c} The closed subspace $U^p_c$ of all continuous
    functions in $U^p$ is a Banach space.
  \end{enumerate}
\end{prop}

\begin{proof}
  Part~\ref{it:u_banach} is straightforward. The embedding $U^p\subset
  U^q$ follows from $\ell^p(\N) \subset \ell^q(\N)$. $U^q\subset
  L^\infty(\R;L^2)$ (including the norm estimate) is obvious for
  atoms, hence also for general $u\in U^q$, and Part~\ref{it:u_emb}
  follows.  This also proves that convergence in $U^q$ implies uniform
  convergence, hence Part~\ref{it:u_c}.  The right-continuity of
  Part~\ref{it:u_right_cont}
  now follows from the definition of atoms. It
  remains to prove~\ref{it:u_limits}: Let $u=\sum_n\lambda_n a_n$ and
  $\eps>0$. There is $n_0\in \N$ such that $\sum_{n\geq n_0+1}
  |\lambda_n|<\eps$. On the one hand, there exists $T_-<0$ such that
  $a_n(t)=0$ for all $t<T_-$, $n=1,\ldots,n_0$, which shows
  $\|u(t)\|_{L^2}<\eps$ for $t<T_-$.  On the other hand, there exists
  $T_+>0$ such that $a_n(t)=a_n(t')$ for all $t,t'>T_+$,
  $n=1,\ldots,n_0$, which implies $\|u(t)-u(t')\|_{L^2}<2\eps$ for
  $t,t'>T_+$.
\end{proof}

The following spaces were introduced by N.~Wiener \cite{Wiener24}.

\begin{definition}\label{def:v}
  Let $1\leq p<\infty$. We define $V^p$ as the normed space of all
  functions $v:\R\to L^2$ such that $v(\infty):=\lim_{t\to
    \infty}v(t)=0$ and $v(-\infty)$ exists and for which the norm
  \begin{equation}\label{eq:hom_norm_v}
    \|v\|_{V^p}:=\sup_{\{t_k\}_{k=0}^K \in \mathcal{Z}}
    \left(\sum_{k=1}^{K}
      \|v(t_{k})-v(t_{k-1})\|_{L^2}^p\right)^{\frac{1}{p}}
  \end{equation}
  is finite.  Likewise, let $V^p_-$ denote the normed space of all
  functions $v:\R\to L^2$ such that $v(-\infty)=0$, $v(\infty)$
  exists, and $\|v\|_{V^p} < \infty$, endowed with the norm
  \eqref{eq:hom_norm_v}.
\end{definition}

\begin{prop}\label{prop:v}Let $1\leq p<q <\infty$.
  \begin{enumerate}
  \item\label{it:v_limits} Let $v:\R\to L^2$ be such that
$$\|v\|_{V_0^p}:=\sup_{\{t_k\}_{k=0}^K \in \mathcal{Z}_0}
\left(\sum_{k=1}^{K}
  \|v(t_{k})-v(t_{k-1})\|_{L^2}^p\right)^{\frac{1}{p}}
$$
is finite. Then, it follows that $v(t_0^+):=\lim_{t \downarrow
  t_0}v(t)$ exists for all $t_0\in [-\infty,\infty)$ and
$v(t_0^-):=\lim_{t \uparrow t_0}v(t)$ exists for all $t_0\in
(-\infty,\infty]$ and moreover,
$$
\|v\|_{V^p}=\|v\|_{V_0^p}.
$$
\item\label{it:v_spaces} We define the closed subspace $V^p_{rc}$
  ($V^p_{-,rc}$) of all right-continuous $V^p$ functions ($V^p_-$
  functions).  The spaces $V^p$, $V^p_{rc}$, $V^p_-$ and $V^p_{-,rc}$
  are Banach spaces.
\item\label{it:v_emb1} The embedding $U^p\subset V_{-,rc}^p$ is
  continuous.
\item\label{it:v_emb2} The embeddings $V^p\subset V^q$ and
  $V^p_-\subset V^q_-$ are continuous.
\end{enumerate}
\end{prop}
\begin{proof}
  Part~\ref{it:v_limits} essentially can be found in \cite{Wiener24},
  \S 1.  Part~\ref{it:v_spaces} is straightforward, the closedness
  follows from the fact that $V^p$ convergence implies uniform
  convergence.  Now, let us prove Part~\ref{it:v_emb1}: Due to
  Proposition~\ref{prop:u}, Part~\ref{it:u_right_cont} and \ref{it:v_emb2}
  it remains to show the norm estimate and it suffices
  to do so for a $U^p$-atom
  $a=\sum_{k=1}^K\mathds{1}_{[t_{k-1},t_k)}\phi_{k-1}$. Let
  $\{s_j\}_{j=1}^J \in \mathcal{Z}$. Then,
  $a(s_j)-a(s_{j-1})=\phi_{k_j-1}-\phi_{k_{j-1}-1}$, which is zero if
  $k_j=k_{j-1}$. It follows
$$
\sum_{j=1}^J\|a(s_j)-a(s_{j-1})\|_{L^2}^p \leq
2^p\sum_{k=1}^K\|\phi_{k-1}\|_{L^2}^p\leq 2^p,
$$
which implies $\|a\|_{V^p} \leq 2$. Part~\ref{it:v_emb2} is implied by
$\ell^p(\N)\subset \ell^q(\N)$.
\end{proof}

\begin{prop}\label{prop:v_dec}
  Let $v\in V^p_{-,rc}$ such that $\|v\|_{V^p}=1$. For any $n \in
  \N_0$
  \begin{enumerate}
  \item there exists $\mathfrak{t}_n \in \mathcal{Z}$ such that
    $\mathfrak{t}_0 \subset \mathfrak{t}_{1}\subset \ldots$ and
    $\#\mathfrak{t}_n\leq 2^{1+np}$,
  \item there exists a right-continuous step-function $u_n$
    subordinate to $\mathfrak{t}_n$ such that $\sup_t
    \|u_n(t)\|_{L^2}\leq 2^{1-n}$,
  \item there exists a $v_n\in V^p_{-,rc}$ such that $\sup_t
    \|v_n(t)\|_{L^2}\leq 2^{-n}$,
  \item it holds $v_n=u_{n+1}+v_{n+1}$, $u_0=0$, $v_0=v$.
  \end{enumerate}
\end{prop}
\begin{proof}
  We proceed by induction: For $n=0$ we define
  $\mathfrak{t}_n:=\{-\infty,\infty\}$, $u_0=0$ and $v_0=v$, hence all
  the claims are immediate.  For $n\in \N$ let
  $\mathfrak{t}_n:=\{-\infty=t_{n,0}<\ldots<t_{n,K_n}\}$ and $u_n,v_n$
  be given with the requested properties. Let $k\in
  \{0,\ldots,K_n-1\}$. For $j=0$ we define $t_{n+1,k}^0:=t_{n,k}$. For
  $j\geq 1$ we define
  \begin{equation*}
    t_{n+1,k}^j:=\inf \{t \mid t^{j-1}_{n+1,k} < t \leq t_{n,k+1}:
    \|v(t)-v(t_{n+1,k}^{j-1})\|_{L^2}>2^{-n-1}\}
  \end{equation*}
  if this set is nonempty and $t_{n+1,k}^j:=t_{n,k+1}$ otherwise.
  
  Now, we relabel all these points $\{t_{n+1,k}^j\}_{j,k}$ as
  \begin{equation*}
    -\infty=t_{n+1,0}<\ldots<t_{n+1,K_{n+1}}=\infty
  \end{equation*}
  which defines the partition $\mathfrak{t}_{n+1} \in \mathcal{Z}$.
  We define
  \begin{align*}
    u_{n+1}:=&\sum_{k=1}^{K_{n+1}}
    \mathds{1}_{[t_{n+1,k-1},t_{n+1,k})}v_n(t_{n+1,k-1})\\
    v_{n+1}:=&v_n-u_{n+1}.
  \end{align*}
  For $t\in \R$ there exists $k$ such that $t\in
  [t_{n+1,k-1},t_{n+1,k})$ and it holds $\|v_{n+1}(t)\|_{L^2}\leq
  \|v_n(t)-v_n(t_{n+1,k-1})\|_{L^2}\leq 2^{-n-1}$. Moreover,
  $1=\|v\|_{V^p}^p\geq (\#t_{n+1}-\#t_{n})2^{-(n+1)p}$ and therefore
  $\#t_{n+1}\leq 2^{1+(n+1)p}$.
\end{proof}

\begin{corollary}\label{cor:emb_v_in_u} Let $1\leq p<q <\infty$.
  The embedding $V^p_{-,rc} \subset U^q$ is continuous.
\end{corollary}
\begin{proof}
  Let $v\in V^p_{-,rc}$ with $\|v\|_{V^p}=1$.
  Then, by Proposition~\ref{prop:v_dec}
  there exist $\mathfrak{t}_n \in \mathcal{Z}$ with
  $\#\mathfrak{t}_n \leq 2^{1+np}$ and associated step-functions $u_n$
  with $\sup_t \|u_n(t)\|_{L^2}\leq 2^{1-n}$ such that
  $v(t)=\sum_{n=0}^\infty u_n(t)$.  Moreover, $\|u_n\|_{U^q}\leq
  4\cdot 2^{n(\frac{p}{q}-1)}$, hence $\sum_{n}\|u_n\|_{U^q}\leq 4
  (1-2^{\frac{p}{q}-1})^{-1}$.  The claim follows since $U^q$ is a
  Banach space.
\end{proof}

\begin{prop}\label{prop:gen_deriv}
  For $u\in U^p$ and $v\in V^{p'}$ and a partition
  $\mathfrak{t}:=\{t_k\}_{k=0}^K \in \mathcal{Z}$ we define
  \begin{align*}
    B_{\mathfrak{t}}(u,v) &:=\sum_{k=1}^{K} \lb u(t_{k-1}),
    v(t_{k})-v(t_{k-1})\rb
  \end{align*}
  Here, $\lb \cdot,\cdot\rb$ denotes the $L^2$ inner product.  There
  is a unique number $B(u,v)$ with the property that for all $\eps>0$
  there exists $\mathfrak{t} \in \mathcal{Z}$ such that for every
  $\mathfrak{t}'\supset\mathfrak{t} $ it holds
  \begin{equation}\label{eq:conv}
    |B_{\mathfrak{t}'}(u,v)-B(u,v)|<\eps,
  \end{equation}
  and the associated bilinear form
  \begin{equation*}
    B:U^p\times V^{p'}: (u,v)\mapsto B(u,v)
  \end{equation*}
  satisfies the estimate
  \begin{equation}\label{eq:b_est}
    |B(u,v)|\leq \|u\|_{U^p}\|v\|_{V^{p'}}.
  \end{equation}
\end{prop}
\begin{proof}
  First of all, we note the following: Let
  $\mathfrak{t}=\{t_n\}_{n=0}^N \in \mathcal{Z}$ and let $u$ be a step
  function $u=\sum_{k=1}^{K} \mathds{1}_{[s_{k-1},s_{k})} \phi_{k-1}$
  subordinate to a partition $\mathfrak{s} \in \mathcal{Z}$ (not
  necessarily an atom), with $\phi_0=0$.  For each $t_n\in
  \mathfrak{t}$, $n<N$, there exists $k_n<K$ such that $s_{k_n}\leq
  t_n <s_{k_n+1}$.  Then,
  \begin{equation}\label{eq:remove1}
    B_{\mathfrak{t}}(u,v)=\sum_{n=1}^N \lb \phi_{k_{n-1}}, v(t_{n})-v(t_{n-1})\rb
  \end{equation}
  Now, if for some $n$ it is $k_{n-1}= k_{n}$, then
  \begin{equation*}
    \lb \phi_{k_{n-1}}, v(t_n)-v(t_{n-1})\rb +\lb \phi_{k_{n}}, v(t_{n+1})-v(t_{n})\rb =
    \lb \phi_{k_{n-1}}, v(t_{n+1})-v(t_{n-1})\rb
  \end{equation*}
  which shows that we may remove such $t_n$ from the partition
  $\mathfrak{t}$ which gives rise to a partition
  $\mathfrak{t}^\ast\subset \mathfrak{t}$. In summary, we may write
  \begin{equation}\label{eq:remove}
    B_{\mathfrak{t}}(u,v)
    =\sum_{n=1}^{N^{\ast}} \lb \phi_{k_{n-1}^\ast}, v(t_n^\ast)-v(t_{n-1}^\ast)\rb
  \end{equation}
  where now $0\leq k_0^\ast<\ldots<k^\ast_{N^\ast-1}< K$.

  Let $\mathfrak{t}\in \mathcal{Z}$ be given.  Assume $a$ is a
  $U^p$-atom.  Obviously, \eqref{eq:remove} and H\"older's inequality
  imply
  \begin{equation*}
    |B_{\mathfrak{t}}(a,v)|\leq \|v\|_{V^{p'}},
  \end{equation*}
  for all $v\in V^{p'}$. Hence,
  \begin{equation*}
    |B_{\mathfrak{t}}(u,v)|\leq \|u\|_{U^p}\|v\|_{V^{p'}},
  \end{equation*}
  for all $u\in U^p$ and $v\in V^{p'}$.

  Now, let $u\in U^p$ and $v\in V^{p'}$ and $\eps>0$.  Let
  $u=\sum_{l=1}^\infty \lambda_l a_l$ be an atomic decomposition such
  that $\sum_{l=n+1}^\infty |\lambda_l|<\eps/(2\|v\|_{V^{p'}})$. We
  define the approximating step function $u_n=\sum_{l=1}^n \lambda_l
  a_l$ and let $\mathfrak{t}\in \mathcal{Z}$ be the subordinate
  partition. Then, for all $\mathfrak{t}'\in \mathcal{Z}$ with
  $\mathfrak{t}\subset \mathfrak{t}'$ it follows as in
  \eqref{eq:remove} that
  \begin{align*}
    |B_{\mathfrak{t}'}(u,v)-B_{\mathfrak{t}}(u,v)|\leq &
    |B_{\mathfrak{t}'}(u,v)-B_{\mathfrak{t}'}(u_n,v)|
    +|B_{\mathfrak{t}}(u_n,v)-B_{\mathfrak{t}}(u,v)| \\
    \leq&2\|u-u_n\|_{U^p}\|v\|_{V^{p'}}\\
    \leq& 2 \sum_{l=n+1}^\infty |\lambda_l|\|v\|_{V^{p'}}<\eps.
  \end{align*}
  Therefore, for a given $j\in \N$ there exists $\mathfrak{t}^{(j)}
  \in \mathcal{Z}$ such that for all $\mathfrak{t}'\in \mathcal{Z}$
  with $\mathfrak{t}^{(j)} \subset \mathfrak{t}'$
  \begin{equation*}
    |B_{\mathfrak{t}'}(u,v)-B_{\mathfrak{t}^{(j)}}(u,v)|<2^{-j},
  \end{equation*}
  and with $t'=\mathfrak{t}^{(j)}\cup\mathfrak{t}^{(j+1)}$ it follows
  \begin{equation*}
    |B_{\mathfrak{t}^{(j+1)}}(u,v)-B_{\mathfrak{t}^{(j)}}(u,v)|<2^{1-j}.
  \end{equation*}
  Hence, $\lim_{j\to \infty} B_{\mathfrak{t}^{(j)}}(u,v)=:B(u,v)$
  exists and \eqref{eq:conv} and \eqref{eq:b_est} are satisfied.
  Property \eqref{eq:conv} also implies the uniqueness.
\end{proof}

\begin{theorem}\label{thm:duality}
  Let $1<p<\infty$.We have
  \begin{equation*}
    (U^p)^\ast=V^{p'}
  \end{equation*}
  in the sense that
  \begin{equation}\label{eq:duality}
    T: V^{p'} \to (U^p)^\ast, \; T(v):=B(\cdot,v)
  \end{equation}
  is an isometric isomorphism.
\end{theorem}
\begin{proof}
  In view of \eqref{eq:b_est} it suffices to show that for each $L\in
  (U^p)^\ast$ there is $v\in V^{p'}$ such that $T(v)(u)=L(u)$ and
  $\|v\|_{V^{p'}}\leq \|L\|$.  To this end, let $0\not=L\in
  (U^p)^\ast$. For $t$ fixed we have $\phi \mapsto - L(
  \mathds{1}_{[t,\infty)}\phi) \in (L^2)^\ast$, hence there exists
  $\widetilde{v}(t)\in L^2$ such that
  $L(\mathds{1}_{[t,\infty)}\phi)=-\lb \phi,\widetilde{v}(t)\rb$ for
  all $\phi \in L^2$.  Fix a partition $\{t_k\}_{k=0}^K \in
  \mathcal{Z}_0$ and define $u:=\sum_{k=1}^K
  \mathds{1}_{[t_{k-1},t_k)}\phi_{k-1}$ with
  \begin{equation*}
    \phi_{k-1}:=\frac{(\widetilde{v}(t_k)-\widetilde{v}(t_{k-1}))
      \|\widetilde{v}(t_k)-\widetilde{v}(t_{k-1})\|_{L^2}^{p'-2}}{
      \left(\sum_{k=1}^{K}\|\widetilde{v}(t_k)-\widetilde{v}(t_{k-1})
        \|^{p'}_{L^2}\right)^{\frac{1}{p}}}.
  \end{equation*}
  Then, $\|u\|_{U^p}\leq 1$ and
  \begin{align*}
    \|L\|&\geq \left|\sum_{k=1}^K
      L(\mathds{1}_{[t_{k-1},t_k)}\phi_{k-1}) \right|
    =\left|\sum_{k=1}^K L(
      \mathds{1}_{[t_{k-1},\infty)}\phi_{k-1})-L(\mathds{1}_{[t_{k},\infty)}\phi_{k-1}
      )\right|\\
    &=\left| \sum_{k=1}^K \lb \phi_{k-1},
      \widetilde{v}(t_{k})-\widetilde{v}(t_{k-1}) \rb \right|
    =\left(\sum_{k=1}^K
      \|\widetilde{v}(t_k)-\widetilde{v}(t_{k-1})\|_{L^2}^{p'}\right)^{\frac{1}{p'}},
  \end{align*}
  which shows that $ \|\widetilde{v}\|_{V_0^{p'}}\leq \|L\| $ and that
  $\lim_{s\to \pm \infty}\widetilde{v}(s)$ exists due to Proposition~\ref{prop:v},
  Part~\ref{it:v_limits}. For
  $v(t):=\widetilde{v}(t)-\widetilde{v}(\infty)$ it follows $v\in
  V^{p'}$ and
  \begin{equation*}
    \|v\|_{V^{p'}}\leq \|L\|.
  \end{equation*}
  It remains to show that $T(v)(u)=L(u)$ for all $u\in U^{p}$: For a
  step function $u=\sum_{k=1}^{K}
  \mathds{1}_{[t_{k-1},t_{k})}\phi_{k-1}$ with underlying partition
  $\mathfrak{t}$ we have
  \begin{align*}
    T(v)(u)=& B_{\mathfrak{t}}(u,v)=
    \sum_{k=1}^K \lb \phi_{k-1} , v(t_{k})-v(t_{k-1})\rb\\
    =&\sum_{k=1}^K L(\mathds{1}_{[t_{k-1},t_{k})} \phi_{k-1})=L(u)
  \end{align*}
  and the claim follows by density and \eqref{eq:b_est}.
\end{proof}

\begin{prop}\label{prop:b0}
  For $1<p<\infty$ let $u\in U^p$ be continuous and $v,v^\ast\in
  V^{p'}$ such that $v(s)=v^\ast(s)$ except for at most countably many
  points.  Then,
  \begin{equation*}
    B(u,v)=B(u,v^\ast).
  \end{equation*}
\end{prop}

\begin{proof}
  For $w:=v-v^\ast$ it holds that $w(s)=0$ except for at most
  countably many points. We have to show that $B(u,w)=0$.  We may
  assume $\|u\|_{U^p}=\|w\|_{V^{p'}}= 1$.  For $\eps>0$ there exists
  $\mathfrak{t}=\{t_k\}_{k=0}^K \in \mathcal{Z}$ such that for every
  $\mathfrak{t}'\supset \mathfrak{t}$:
  \begin{equation*}
    |B_{\mathfrak{t}'}(u,w)-B(u,w)|<\eps.
  \end{equation*}
  Since $u$ is continuous, there exists $\delta>0$ such that for all
  $k\in \{1,\ldots,K-1\}$ and $s\in (t_k-\delta,t_k)$ it holds
  $\|u(s)-u(t_k)\|_{L^2}<\frac\eps{K}$.  For all
  $k\in\{1,\ldots,K-1\}$ we choose $t_k^\ast\in (t_k-\delta,t_k)$ such
  that $t_k^\ast>t_{k-1}$ and $w(t_k^\ast)=0$ and set
  \begin{equation*}
    \mathfrak{t}'=\mathfrak{t}\cup \{t_1^\ast,\ldots,t_{K-1}^\ast\}.
  \end{equation*}
  Summation by parts yields
  \begin{equation*}
    B_{\mathfrak{t}'}(u,w) = \sum_{k=1}^{K-1}\lb u(t_k^\ast)-u(t_k),w(t_k)\rb.
  \end{equation*}
  Hence, $|B(u,w)|<|B_{\mathfrak{t}'}(u,w)|+\eps<2\eps$.
\end{proof}

\begin{prop}\label{prop:c1} Let $1<p<\infty$, $u\in V^1_-$ be absolutely 
  continuous on compact intervals and $v\in V^{p'}$. Then,
  \begin{equation}\label{eq:c1}
    B(u,v)=-\int_{-\infty}^\infty \lb u'(t),v(t)\rb dt.
  \end{equation}
\end{prop}
\begin{proof}
  Without loss of generality we may assume
  $\|u\|_{V^1}=\|v\|_{V^{p'}}= 1$.  By Corollary~\ref{cor:emb_v_in_u}
  we have $u\in U^p$, so that the left hand side of \eqref{eq:c1}
  makes sense.  From our assumptions on $u$ it follows that $u'\in
  L^1(\R;L^2)$ with $\|u'\|_{L^1}\leq \|u\|_{V^1}=1$ and that the
  Fundamental Theorem of Calculus is valid (cf.\ for example
  \cite{Federer96}, Corollary~2.9.20 and 2.9.22).  Because $u$ is
  continuous and $v$ is left-continuous except for at most countably
  many points, it suffices by Proposition~\ref{prop:b0} to consider
  left-continuous $v\in V^{p'}$.  For $\eps>0$ there exists
  $\mathfrak{t}=\{t_n\}_{n=0}^{N}\in \mathcal{Z}$ such that for every
  $\mathfrak{t}'\supset \mathfrak{t}$ the estimate \eqref{eq:conv} is
  satisfied.  Furthermore, there exists $T_1\leq t_1$ and $T_2\geq
  t_{N-1}$ such that $\|v(t)-v(T_1)\|_{L^2}<\eps$ for $t\leq T_1$ and
  $\|v(t)\|_{L^2}<\eps$ for $t\geq T_2$.  Since $v$ is a
  left-continuous, regulated function on $[T_1,T_2]$, there exists
  $\mathfrak{t}'=\{t_n'\}_{n=0}^{N'} \supset \mathfrak{t}$ such that
  $t_1'=T_1$ and $t_{N'-1}'= T_2$ and
  \begin{equation*}
    \|v-w\|_{L^\infty} < \eps,
    \text{ for } w:=\sum_{n=1}^{N'-1} v(t'_{n})\mathds{1}_{(t'_{n-1},t'_n]}
  \end{equation*}
  Now, estimate \eqref{eq:conv} and summation by parts yield
  \begin{equation*}
    |-\sum_{n=1}^{N'-1} \lb u(t_{n}')-u(t_{n-1}'),v(t'_n))\rb-B(u,v)| < \eps.
  \end{equation*}
  By the Fundamental Theorem of Calculus and the definition of $w$ we
  have for $n\in \{1,\ldots,N'-1\}$:
  \begin{equation*}
    \lb u(t_{n}')-u(t_{n-1}'),v(t'_n))\rb=\int_{t_{n-1}'}^{t_{n}'}
    \lb u'(s),w(s)\rb ds.
  \end{equation*}
  Altogether, we obtain
  \begin{equation*}
    \left|-\int_{-\infty}^{\infty} \lb u'(s),v(s))\rb ds-B(u,v)\right|
    < \|u'\|_{L^1} \|v-w\|_{L^\infty} + \eps < 2\eps,
  \end{equation*}
  which finishes the proof.
\end{proof}

\begin{remark}\label{rem:sub} For $u\in U^p$ it is
clear that
\begin{equation*}
\|u\|_{U^p}=\sup_{v \in V^{p'}: \|v\|_{V^{p'}=1}} |B(u,v)|
\end{equation*}
by Theorem~\ref{thm:duality}.
Although we will not use it in the sequel, let us remark that
for $u\in V^1_{-}$ which is absolutely continuous on compact
intervals it holds
\begin{equation*}
\|u\|_{U^p}
=\sup_{v \in V^{p'}_c: \|v\|_{V^{p'}=1}} |B(u,v)| \, ,
\end{equation*}
where $V^{p'}_c$ is the set of all
continuous functions in $V^{p'}$ (which is obviously not dense).
This may be seen as follows: By
Proposition~\ref{prop:c1} we may restrict the supremum to $V^{p'}_{rc}$. Then,
we may restrict this further to the dense subset of the
right-continuous step-functions $\mathcal{T}_{rc}$. Finally, we may replace
$\mathcal{T}_{rc}$ by $V^{p'}_c$ by substituting jumps
in a piecewise linear and continuous way with the help of \eqref{eq:c1}.
\end{remark}

\begin{remark}\label{rem:dual2}
For $v \in V^p$ Theorem~\ref{thm:duality} also implies
\begin{equation*}
\|v\|_{V^p}=\sup_{u \text{ $U^{p'}$-atom}}|B(u,v)|
\end{equation*}
for  $1<p<\infty$.
\end{remark}

We will use the convention that capital letters denote dyadic numbers,
e.g. $N=2^n$ for $n \in \Z$ and for a dyadic summation we write
$\sum_N a_N :=\sum_{n \in \Z} a_{2^n}$ and $\sum_{N\geq M} a_N
:=\sum_{n \in \Z: 2^n \geq M} a_{2^n}$ for brevity.  Let $\chi \in
C^\infty_0((-2,2))$ be an even, non-negative function such that
$\chi(t)=1$ for $|t|\leq 1$. We define $\psi(t):=\chi(t)-\chi(2t)$ and
$\psi_N:=\psi(N^{-1}\cdot)$.  Then, $\sum_{N}\psi_N(t)=1$ for $t\neq
0$. We define
\begin{equation*}
  \widehat{Q_N u}:=\psi_N \widehat{u}
\end{equation*}
and $\widehat{Q_0u}=\chi(2\cdot) \widehat{u}$, $Q_{\geq M}=\sum_{N\geq
  M} Q_N$ as well as $Q_{<M}=I-Q_{\geq M}$.

\begin{definition}\label{def:besov}
  Let $s\in \R$ and $1\leq p,q \leq \infty$. We define the semi-norms
  \begin{equation}\label{eq:besov_norm}
    \begin{split}
      \|u\|_{\hb s p q}&:= \left( \sum_{N} N^{qs} \|Q_N
        u\|^q_{L^p(\R;L^2))}
      \right)^{\frac{1}{q}}\quad (q<\infty)\\
      \|u\|_{\hb s p \infty}&:=
      \sup_{N} N^{s} \|Q_N u\|_{L^p(\R;L^2)}\\
    \end{split}
  \end{equation}
  for all $u\in \mathcal{S}'(\R;L^2)$ for which these numbers are
  finite.
\end{definition}

\begin{prop}\label{prop:besov_emb} Let $1< p<\infty$.
  For any $v\in V^p$, the estimate
  \begin{equation}\label{eq:besov_emb1}
    \|v\|_{\hb {\frac1p} p  {\infty}} \ls \|v\|_{V^p}
  \end{equation}
  holds true. Moreover, for any $u \in \mathcal{S}'(\R;L^2)$ such that
  the semi-norm $\|u\|_{\hb {\frac1p} p 1}$ is finite there exists
  $u(\pm \infty)\in L^2$. Then, $u-u(- \infty)\in U^p$ and the
  estimate
  \begin{equation}
    \label{eq:besov_emb2}
    \|u-u(-\infty)\|_{U^p} \ls \|u\|_{\hb {\frac1p} p 1}
  \end{equation}
  holds true.
\end{prop}
\begin{proof}
  Concerning \eqref{eq:besov_emb1}, see e.g. Example 9 in
  \cite{Peetre76}, pp.  167-168.  Now, the second part follows by
  duality: Let $u \in \mathcal{S}'(\R;L^2)$ such that $\|u\|_{\hb
    {\frac1p} p 1}<\infty$ and we consider $Q_N u \in L^p(\R;L^2)$,
  which is smooth. Hence, $Q_N u\in U^p$.  Then,
  \begin{align*}
    \|Q_N u\|_{U^p}=&\sup_{\|L\|_{(U^p)^*}=1} |L(Q_Nu)|
    =\sup_{\|v\|_{V^{p'}}=1}|B(Q_N u,v)|\\
    =&\sup_{\|v\|_{V^{p'}}=1} \left|\int_{-\infty}^\infty \lb Q_N
      u'(t),v(t)\rb dt
    \right| \\
    \leq&\sup_{\|v\|_{V^{p'}}=1}\|Q_N u\|_{\hb {\frac1p} p
      1}\|v\|_{\hb {\frac1{p'}} {p'} \infty} \ls N^{{\frac1p}}\|Q_N
    u\|_{L^p},
  \end{align*}
  and it follows that $\widetilde{u}:=\sum_N Q_N u$ converges in $U^p$
  and $\|\widetilde{u}\|_{U^p}\ls \|u\|_{\hb {\frac1p} p 1}$. It is
  $\|u-\widetilde{u}\|_{\hb {\frac1p} p 1}=0$, hence
  $u=\widetilde{u}+const$ and the claim follows.
\end{proof}

Now, we focus on the spatial dimension $d=2$ (i.e. $L^2=L^2(\R^2;\C)$)
and consider $S:=-\partial^3_{x}-\partial_x^{-1}\partial_y^2$.  We
define the associated unitary operator $e^{tS}:L^2\to L^2$ to be the
Fourier multiplier
\begin{equation*}
  \widehat{e^{tS}u_0}(\xi,\eta)
  =\exp(it(\xi^3-\frac{\eta^2}{\xi}))\widehat{u_0}(\xi,\eta).
\end{equation*}

\begin{definition}\label{def:S_spaces}
  We define
  \begin{enumerate}
  \item\label{it:ul} $U^p_S=e^{\cdot S}U^p$ with norm
    $\|u\|_{U^p_S}=\|e^{-\cdot S} u\|_{U^p}$ ,
  \item\label{it:vl} $V^p_S=e^{\cdot S}V^p$ with norm
    $\|v\|_{V^p_S}=\|e^{-\cdot S} u\|_{V^p}$ ,
  \end{enumerate}
  and similarly the closed subspaces $U^p_{c,S}$, $V^p_{rc,S}$,
  $V^p_{-,S}$ and $V^p_{-,rc,S}$.
\end{definition}
Let us note that for $S$ defined above these spaces are invariant under komplex conjugation,
because the symbol of $S$ is an odd function.

Let us define the smooth projections
\begin{align*}
  \widehat{P_N u}(\tau,\xi,\eta)&:=\psi_N(\xi)
  \widehat{u}(\tau,\xi,\eta)
  ,\\
  \widehat{Q^S_M u}(\tau,\xi,\eta)&:=\psi_M(\tau-\xi^3+\eta^2\xi^{-1})
  \widehat{u}(\tau,\xi,\eta),
\end{align*}
as well as $\widehat{P_0 u}(\tau,\xi,\eta) := \chi(2\xi)
\widehat{u}(\tau,\xi,\eta)$, $Q^S_{\geq M}:=\sum_{N \geq M} Q^S_N$,
and $Q^S_{<M}:=I-Q^S_{\geq M}$.  Note that we have
\begin{equation}\label{eq:relq}
  Q^S_M = e^{\cdot S} Q_M e^{-\cdot S}
\end{equation}
and similarly for $Q^S_{\geq M}$ and $Q^S_{<M}:=I-Q^S_{\geq M}$.

\begin{definition}\label{def:x}
  Let $s,b \in \R$ and $1\leq q \leq \infty$. We define the semi-norms
  \begin{equation}\label{eq:x_norm}
    \|u\|_{\hx s b q}
    :=\left( \sum_{N}N^{2s}\|e^{-\cdot S}P_N u\|^2_{\hb b 2 q}\right)^{\frac12}
  \end{equation}
  for all $u\in \mathcal{S}'(\R;L^2)$ for which these numbers are
  finite.
\end{definition}

Now, we may identify $u\in\mathcal{S}'(\R;L^2)$ with a subset of $
\mathcal{S}'(\R^3)$ and
\begin{equation*}
  \|u\|_{\hx s b q}=\left(\sum_{N}N^{2s} \left(\sum_{M}M^{bq} \|P_NQ^S_Mu\|_{L^2(\R^3)}^q
    \right)^{\frac2q}\right)^{\frac12}
\end{equation*}
with the obvious modification in the case $q=\infty$.

\begin{corollary}\label{cor:mod_est}
  We have
  \begin{align}\label{eq:mod_est1}
    \|Q^S_M u\|_{L^2(\R^3)} &\ls M^{-\frac{1}{2}} \|u\|_{V^2_S}\\
    \label{eq:mod_est2}
    \|Q^S_{\geq M} u\|_{L^2(\R^3)} &\ls M^{-\frac{1}{2}} \|u\|_{V^2_S}\\
    \label{eq:mod_est3}
    \|Q^S_{< M} u\|_{V^p_S} \ls \|u\|_{V^p_S}\ ,\quad & \|Q^S_{\geq M} u\|_{V^p_S}
    \ls \|u\|_{V^p_S}\\
    \label{eq:mod_est4}
    \|Q^S_{< M} u\|_{U^p_S} \ls \|u\|_{U^p_S}\ ,\quad &\|Q^S_{\geq M}
    u\|_{U^p_S} \ls \|u\|_{U^p_S}
  \end{align}
\end{corollary}
\begin{proof}
  By \eqref{eq:relq} and Definition~\ref{def:S_spaces}, we see that
  \eqref{eq:mod_est1} follows from
  \begin{equation}\label{eq:mod_est1c}
    \|Q_M v\|_{L^2(\R^3)} \ls M^{-\frac12}\|v\|_{V^2}
  \end{equation}
  and similarly for \eqref{eq:mod_est2} -- \eqref{eq:mod_est4}.  Now,
  \eqref{eq:mod_est1c} is just a reformulation of the Besov embedding
  \eqref{eq:besov_emb1}.  Furthermore, \eqref{eq:mod_est1c} implies
  that
  \begin{equation*}
    \|Q_{\geq M} v\|_{L^2(\R^3)}
    \ls \|v\|_{V^2} \sum_{N\geq M} N^{-\frac{1}{2}}
  \end{equation*}
  and \eqref{eq:mod_est2} follows from $\sum_{N \geq M}
  N^{-\frac12}\ls M^{-\frac12}$.  We only need to prove the left
  inequalities of \eqref{eq:mod_est3} and \eqref{eq:mod_est4} because
  of $Q_{\geq M} = I-Q_{<M}$.  By scaling it suffices to show
  \eqref{eq:mod_est3} and \eqref{eq:mod_est4} for $M=1$ only.  We have
  $Q_{<1}v = \phi*v$ for some Schwartz function $\phi$.  Due to the
  Riemann-Lebesgue Lemma, $Q_{<1}(\pm \infty)=0$.  For
  $\{t_k\}_{k=0}^K \in \mathcal{Z}_0$ we apply Minkowski's inequality
  \begin{align*}
    &\left( \sum_{k=1}^K \|Q_{<1}v(t_k)-Q_{<1}v(t_{k-1})\|_{L^2}^p \right)^{\frac1p}\\
    & \leq \left( \sum_{k=1}^K (\int_{\R} |\phi(s)|
      \|v(t_k-s)-v(t_{k-1}-s)\|_{L^2} ds)^p
    \right)^{\frac1p}\\
    & \leq \int_{\R} |\phi(s)| \left( \sum_{k=1}^K
      \|v(t_k-s)-v(t_{k-1}-s)\|_{L^2}^p \right)^{\frac1p} ds \leq
    \|\phi\|_{L^1(\R)} \|v\|_{V^p}
  \end{align*}
  which implies \eqref{eq:mod_est3}.  Let us finally prove
  \eqref{eq:mod_est4}:
  \begin{equation*}
    \|Q_{<1} u\|_{U^p}=\sup_{\|L\|_{(U^p)^*}=1} |L(\phi*u)|
    =\sup_{\|v\|_{V^{p'}}=1}|B(\phi*u,v)|
  \end{equation*}
  with $\phi$ as above.  For given $\mathfrak{t}=\{t_k\}_{k=0}^K \in
  \mathcal{Z}$ we obtain
  \begin{align*}
    |B_{\mathfrak{t}}(\phi*u,v)| & \leq \left|\sum_{k=1}^{K-1} \lb
      (\phi*u)(t_{k-1}) , v(t_{k})-v(t_{k-1})\rb \right| \\ & \leq
    \int_{\R} |\phi(s)| \left|\sum_{k=1}^{K-1}\lb u(t_{k-1}-s) ,
      v(t_k)-v(t_{k-1}) \rb \right| ds \\ & \leq \|\phi\|_{L^1(\R)}
    \|u\|_{U^p} \|v\|_{V^{p'}}.
  \end{align*}
  Since this bound is independent of $\mathfrak{t}$,
  \eqref{eq:mod_est4} follows.
\end{proof}

Similarly to \cite{KT07}, Corollary 3.3 or \cite{Tao07}, Lemma 4.1 we
have the following general extension result, which is well-known at
least for Bourgain type spaces (cp. \cite{GTV97}, Lemma 2.3):
\begin{prop}\label{prop:ext_free_est}
  Let $$T_0:L^2\times \cdots \times L^2\to L^1_{loc}(\R^2;\C)$$ be a
  $n$-linear operator.
  \begin{enumerate}
  \item\label{it:time_space} Assume that for some $1\leq
    p,q\leq\infty$
    \begin{equation*}
      \|T_0(e^{\cdot S}\phi_1,\ldots,e^{\cdot
        S}\phi_n)\|_{L^p_t(\R;L^q_{x,y}(\R^2))}
      \ls \prod_{i=1}^n \|\phi_i\|_{L^2}.
    \end{equation*}
    Then, there exists $T:U^p_S\times \cdots \times U^p_S \to
    L^p_t(\R;L^q_{x,y}(\R^2))$ satisfying
    \begin{equation*}
      \|T(u_1,\ldots,u_n)\|_{L^p_t(\R;L^q_{x,y}(\R^2))}\ls \prod_{i=1}^n \|u_i\|_{U_S^p},
    \end{equation*}
    such that $
    T(u_1,\ldots,u_n)(t)(x,y)=T_0(u_1(t),\ldots,u_n(t))(x,y) $ a.e..
  \item\label{it:space_time} Assume that for some $1\leq
    p,q\leq\infty$
    \begin{equation*}
      \|T_0(e^{\cdot S}\phi_1,\ldots,e^{\cdot
        S}\phi_n)\|_{L^q_x(\R;L^p_{t,y}(\R^2))}
      \ls \prod_{i=1}^n \|\phi_i\|_{L^2}.
    \end{equation*}
    For $r:=\min(p,q)$ there exists $T:U^r_S\times \cdots \times U^r_S
    \to L^q_x(\R;L^p_{t,y}(\R^2))$ satisfying
    \begin{equation*}
      \|T(u_1,\ldots,u_n)\|_{L^q_x(\R;L^p_{t,y}(\R^2))}\ls \prod_{i=1}^n \|u_i\|_{U_S^r},
    \end{equation*}
    such that $
    T(u_1,\ldots,u_n)(x)(t,y)=T_0(u_1(t),\ldots,u_n(t))(x,y) $ a.e..
  \end{enumerate}
\end{prop}
\begin{proof} Concerning Part~\ref{it:time_space}, we define
  \begin{equation*}
    T(u_1,\ldots,u_n)(t)(x,y)=T_0(u_1(t),\ldots,u_n(t))(x,y).
  \end{equation*}
  Let $a_1,\ldots,a_n$ be $U^p_S$-atoms given as
  \begin{equation*}
    a_i=\sum_{k_i=1}^{K_i}\mathds{1}_{[t_{k_i-1,i},t_{k_i,i})}e^{\cdot S}\phi_{k_i-1,i}
  \end{equation*}
  such that $\sum_{k_i=1}^{K_i}\|\phi_{k_i-1,i}\|^p_{L^2}=1$ and
  $\phi_{0,i}=0$. Then, we use H\"older's inequality
  \begin{align*}
    &\|T(a_1,\ldots,a_n)\|_{L^p_t(\R;L^q_{x,y}(\R^2))}\\
    \leq & \left\| \sum_{k_1,\ldots, k_n}
      \prod_{i=1}^n\mathds{1}_{[t_{k_i-1,i},t_{k_i,i})}
      \left\|T_0(e^{tS}\phi_{k_1-1,1},\ldots,e^{tS}\phi_{k_n-1,n})\right\|_{L^q_{x,y}(\R^2)}
    \right\|_{L^p_t(\R)}\\
    \leq & \left(\sum_{k_1,\ldots,k_n}
      \left\|T_0(e^{tS}\phi_{k_1-1,1},\ldots,e^{tS}\phi_{k_n-1,n})
      \right\|^p_{L^p_t(\R;L^q_{x,y}(\R^2))}\right)^{\frac1p}\\
    \ls &\left(\sum_{k_1,\ldots,k_n}
      \prod_{i=1}^n\|\phi_{k_i-1,i}\|^p_{L^2(\R^2)}\right)^{\frac1p}=1
  \end{align*}
  and the claim follows.

  Now, we turn to the proof of Part~\ref{it:space_time}: We define
  \begin{equation*}
    T(u_1,\ldots,u_n)(x)(t,y)=T_0(u_1(t),\ldots,u_n(t))(x,y).
  \end{equation*}
  Let $a_1,\ldots,a_n$ be $U^r_S$-atoms for $r=\min(p,q)$. Then, by
  H\"older's and Minkowski's inequality (here, we use $r\leq p,q$)
  \begin{align*}
    &\|T(a_1,\ldots,a_n)\|_{L^q_x(\R;L^p_{t,y}(\R^2))}\\
    \leq &\left\| \left(\sum_{k_1,\ldots, k_n}
        |T_0(e^{tS}\phi_{k_1-1,1},\ldots,e^{tS}\phi_{k_n-1,n})|^r\right)^{\frac1r}
    \right\|_{L^q_x(\R;L^p_{t,y}(\R^2))}\\
    \ls & \left(\sum_{k_1,\ldots, k_n}
      \left\|T_0(e^{tS}\phi_{k_1-1,1},\ldots,e^{tS}\phi_{k_n-1,n})
      \right\|_{L^q_x(\R;L^p_{t,y}(\R^2))}^r\right)^{\frac1r}\\
    \ls &\left(\sum_{k_1,\ldots,k_n}
      \prod_{i=1}^n\|\phi_{k_i-1,i}\|^r_{L^2(\R^2)}\right)^{\frac1r}=1
  \end{align*}
  and the claim follows.
\end{proof}

\begin{prop}\label{prop:interpolation} Let $q>1$,
  $E$ be a Banach space and $T:U_S^q\to E$ be a bounded, linear
  operator with $\|Tu\|_E \leq C_q \|u\|_{U^q_S}$ for all $u \in
  U^q_S$.  In addition, assume that for some $1\leq p <q$ there exists
  $C_p\in (0,C_q]$ such that the estimate $\|Tu\|_E \leq C_p
  \|u\|_{U^p_S}$ holds true for all $u \in U^p_S$.  Then, $T$
  satisfies the estimate
  \begin{equation*}
    \|Tu\|_E \leq
    \frac{4 C_p}{\alpha_{p,q}}
    (\ln\frac{C_q}{C_p}+2\alpha_{p,q}+1)\|u\|_{V^p_S},
    \quad u \in V^p_{-,rc,S}
  \end{equation*}
  where $\alpha_{p,q}=(1-\frac{p}{q})\ln(2)$.
\end{prop}
\begin{proof}
  Let $v \in V^p_{-,rc,S}$ be such that $\|v\|_{V^p_S}=1$.  Due to
  Proposition~\ref{prop:v_dec} there exists $u_n\in U^r$ for all
  $r\geq 1$ such that $v=\sum_{n=1}^\infty u_n$ in $U^q$ and
  $\|u_n\|_{U^r_S}\leq 4\cdot 2^{n(\frac{p}{r}-1)}$. For $N\in \N$ it
  follows $\|\sum_{n=1}^N u_n\|_{U^p_S}\leq 4N$ and
  $\|\sum_{n=N+1}^\infty u_n\|_{U^q_S}\leq 4\cdot
  2^{-N(1-\frac{p}{q})}$.  We obtain the estimate
  \begin{equation*}
    \|Tv\|_E\leq 4 C_p N + 4 C_q 2^{-N(1-\frac{p}{q})}.
  \end{equation*}
  Minimizing with respect to $N\in \N$ gives the desired upper bound.
\end{proof}

\begin{corollary}\label{cor:disp_est}
  We have
  \begin{align}
    \label{eq:strichartz}
    \|u\|_{L^4(\R^3)} &\ls \|u\|_{U^4_S}\\
    \label{eq:strichartz_v}
    \|u\|_{L^4(\R^3)} &\ls \|u\|_{V^p_{-,S}} \quad (1\leq p <4)\\
    \label{eq:loc_smooth}
    \|\partial_x u\|_{L^\infty_x(\R;L^2_{t,y}(\R^2))} &\ls \|u\|_{U^2_S}\\
    \label{eq:bil_strichartz1}
    \|P_{N_1}u_1P_{N_2}u_2\|_{L^2(\R^3)}&\ls
    \left(\frac{N_1}{N_2}\right)^{\frac{1}{2}}
    \|P_{N_1}u_1\|_{U^2_S}\|P_{N_2}u_2\|_{U^2_S}
  \end{align}
  Moreover, for $N_2\geq N_1$ and $u_1,u_2\in V^2_{-,S}$ it holds
  \begin{equation}
    \label{eq:bil_strichartz2}
    \|P_{N_1}u_1P_{N_2}u_2\|_{L^2(\R^3)}\ls \left(\frac{N_1}{N_2}\right)^{\frac{1}{2}}
    \left(\ln\left(\frac{N_2}{N_1}\right)+1\right)^2
    \|P_{N_1}u_1\|_{V^2_S}\|P_{N_2}u_2\|_{V^2_S}.
  \end{equation}
\end{corollary}
\begin{proof}
  Proposition 2.3 of \cite{Saut93} and Lemma 3.2 of
  \cite{KenigZiesler05} show that the estimates \eqref{eq:strichartz}
  and \eqref{eq:loc_smooth} hold true for free solutions.  Thus, the
  claims \eqref{eq:strichartz} and \eqref{eq:loc_smooth} follow from
  Proposition~\ref{prop:ext_free_est}. Then, \eqref{eq:strichartz_v} follows from
  Corollary~\ref{cor:emb_v_in_u} and the observation that
  $v\in V^p_{-,S}$ coincides a.e.\ with its right-continuous variant.
  In order to prove
  \eqref{eq:bil_strichartz1}, let $u_i=e^{tS}\phi_i$ ($i=1,2$) be free
  solutions, $\phi_i \in L^2(\R^2)$. With the smooth cutoff in time
  $\chi$ we obtain
  \begin{align*}
    & \|P_{N_1}u_1 P_{N_2}u_2\|_{L^2([-1,1]\times\R^2)}\\\leq & \|\chi
    P_{N_1}u_1 \chi P_{N_2}u_2\|_{L^2(\R^3)}\ls
    \left(\frac{N_1}{N_2}\right)^{\frac{1}{2}}
    \|P_{N_1}\phi_1\|_{L^2}\|P_{N_2}\phi_2\|_{L^2}
  \end{align*}
  which is an immediate consequence of \cite{Hadac07}, Theorem 3.3. By
  rescaling it follows
  \begin{equation*}
    \|P_{N_1}u_1 P_{N_2}u_2\|_{L^2(\R^3)}\ls
    \left(\frac{N_1}{N_2}\right)^{\frac{1}{2}}
    \|P_{N_1}\phi_1\|_{L^2}\|P_{N_2}\phi_2\|_{L^2}
  \end{equation*}
  and we may apply Proposition~\ref{prop:ext_free_est}.

  Now, the estimate \eqref{eq:bil_strichartz2} follows from
  interpolation between \eqref{eq:strichartz} and
  \eqref{eq:bil_strichartz1} via Proposition~\ref{prop:interpolation}
  and again replace $v\in V^p_{-,S}$ by its right-continuous variant.
\end{proof}

\begin{definition}\label{def:res_space}Let $s\leq 0$.
  \begin{enumerate}
  \item\label{it:hom_res_space_v} Define $\dot Y^s$ as the closure of
    all $u\in C(\R;H^{1,1}(\R^2))\cap V_{-,rc,S}^2$ such that
    \begin{equation}\label{eq:hom_res_space_v}
      \|u\|_{\dot Y^s}:= \left(\sum_{N} N^{2s} \|P_N u\|_{V^2_S}^2\right)^{\frac{1}{2}}<\infty,
    \end{equation}
    in the space $C(\R;\dot H^{s,0}(\R^2))$ with respect to the
    $\|\cdot\|_{\dot Y^s}$-norm.
  \item\label{it:hom_res_space} Define $\dot Z^s$ as the closure of
    all $u\in C(\R;H^{1,1}(\R^2))\cap U_{S}^2$ such that
    \begin{equation}\label{eq:hom_res_space}
      \|u\|_{\dot Z^s}:= \left(\sum_{N} N^{2s} \|P_N u\|_{U^2_S}^2\right)^{\frac{1}{2}}<\infty,
    \end{equation}
    in the space $C(\R;\dot H^{s,0}(\R^2))$ with respect to the
    $\|\cdot\|_{\dot Z^s}$-norm.
  \item\label{it:inhom_res_space} Define $X$ as the closure of all
    $u\in C(\R;H^{1,1}(\R^2))\cap U^2_S$ such that
    \begin{equation}\label{eq:inhom_res_space1}
      \|u\|_{X}:=\|u\|_{\dot Z^0}+\|u\|_{\hx 0  1  1}<\infty,
    \end{equation}
    in the space $C(\R;L^2(\R^2))$ with respect to the
    $\|\cdot\|_{X}$-norm.  Define $Z^s:=\dot Z^s+X$, with norm
    \begin{equation}\label{eq:inhom_res_space}
      \|u\|_{Z^s}=\inf\{\|u_1\|_{\dot Z^s}+\|u_2\|_X \mid u=u_1+u_2\}.
    \end{equation}
  \end{enumerate}
\end{definition}

\begin{remark}\label{rem:restr}
  Let $E$ be a Banach space of continuous functions $f:\R\to H$, for
  some Hilbert space $H$.  We also consider the corresponding
  restriction space to the interval $I\subset \R$ by
  \begin{equation*}
    E(I)=\{u\in C(I,H)\mid \exists \ \widetilde{u}\in E: \
    \widetilde{u}(t)=u(t),\ t \in I \}
  \end{equation*}
  endowed with the norm $\|u\|_{E(I)}=\inf\{\|\widetilde{u}\|_E \mid
  \tilde{u}:\widetilde{u}(t)=u(t), t \in I\}$.  Obviously, $E(I)$ is
  also a Banach space.
\end{remark}

\begin{prop}\label{prop:cont_no}
  \begin{enumerate}
  \item\label{it:cont_no1} Let $T>0$ and $u\in \dot Y^{s}([0,T])$,
    $u(0)=0$. Then, for every $\eps>0$ there exists $0\leq T'\leq T$
    such that $\|u\|_{\dot Y^{s}([0,T'])}<\eps$.
  \item\label{it:cont_no2} Let $T>0$ and $u\in \dot Z^{s}([0,T])$,
    $u(0)=0$. Then, for every $\eps>0$ there exists $0\leq T'\leq T$
    such that $\|u\|_{\dot Z^{s}([0,T'])}<\eps$.
  \end{enumerate}
\end{prop}
\begin{proof} It is enough to consider $s=0$. Assume $u\in \dot
  Y^{0}([0,T])$ with $u(0)=0$ and let $\widetilde{u}\in \dot Y^{0}$ be
  an extension. There exists a decomposition
  $\widetilde{u}=\widetilde{u}_h+\widetilde{u}_r$ with
  \begin{equation} \label{eq:dec} \widetilde{u}_h = \sum_{M_0\leq
      N\leq M_1} P_N \widetilde{u},\quad \| \widetilde{u}_r
    \|_{\dot{Y}^0} < \eps.
  \end{equation}
  Due to the right-continuity of $\widetilde{u}_h$ there exists
  $0<T_0\leq T$ with
  $\|\widetilde{u}_h\|_{L^\infty([0,T_0];L^2)}<\eps$.  Moreover, there
  exists $\mathfrak{t}=\{t_k\}_{k=0}^K\in \mathcal{Z}$ such that $0\in
  \mathfrak{t}$ and
  \begin{equation*}
    \left(\sum_{k=1}^{K} \|e^{-t_{k}S}\widetilde{u}_h(t_{k})-
      e^{-t_{k-1}S}\widetilde{u}_h(t_{k-1})\|^2_{L^2}\right)^{\frac12}
    >\|\widetilde{u}_h\|_{V^2_S}-\eps.
  \end{equation*}
  We define $T':=\min\{t_k\mid t_k>0\}$ and the continuous extension
  \begin{equation} \label{eq:cont_ext} \widetilde{u}_{h,T'} :=
    \mathds{1}_{[0,T')} \widetilde{u}_h + \mathds{1}_{[T',\infty)}
    \widetilde{u}_h(T').
  \end{equation}
  Then, $\|\widetilde{u}_{h,T'}\|_{V^2_S}<\eps$.
  Finally,
  \begin{equation*}
    \|u_h\|_{\dot{Y}^0([0,T'])} \leq \|\widetilde{u}_{h,T'}\|_{\dot{Y}^0} \leq 
    \left( 
      \sum_{M_0/2\leq N \leq 2M_1} \|P_N \widetilde{u}_{h,T'}\|_{V^2_S}^2
    \right)^{\frac12} \ls \eps.
  \end{equation*}

  In order to prove Part~\ref{it:cont_no2} let us assume that $u\in
  \dot Z^{0}([0,T])$ with $u(0)=0$ and let $\widetilde{u}\in \dot
  Z^{0}$ be an extension. We perform a similar decomposition as in
  \eqref{eq:dec}.  Since $\widetilde{u}_h\in U^2_S$, we have an atomic
  decomposition
  \begin{equation*}
    \widetilde{u}_h=\sum_{k=1}^\infty \lambda_k e^{\cdot S} a_k \quad \text{s.th. }
    \sum_{k=k_0+1}^{\infty}|\lambda_k| <\eps.
  \end{equation*}
  There exists $0<T'\leq T$, such that all $a_k$ ($k=1,\ldots, k_0$)
  are constant on $[0,T']$.  Define $\lambda_0=\|\sum_{k=1}^{k_0}
  \lambda_k a_k(0)\|_{L^2}$ and $\phi=\lambda_0^{-1}\sum_{k=1}^{k_0}
  \lambda_k a_k(0)$ as well as the atom
  $a_0=\mathds{1}_{[0,\infty)}\phi$. Then,
  \begin{equation*}
    \lambda_0
    =\left\|u(0)-\sum_{k=k_0+1}^{\infty} \lambda_k a_k(0)\right\|_{L^2}
    \leq \sum_{k=k_0+1}^{\infty}|\lambda_k|< \eps.
  \end{equation*}
  For $f(t):=\lambda_0 e^{tS} a_0(t)+\sum_{k=k_0+1}^{\infty} \lambda_k
  e^{tS} a_k(t)$, we define the continuous function
  $f_{T'}=\mathds{1}_{[0,T')}f+\mathds{1}_{[T',\infty)}f(T'-)$.  It
  holds $u_h(t)=\widetilde{u}_h(t)=f_{T'}(t)$ for $t\in [0,T']$ and
  therefore $\|u_h\|_{\dot{Z}^0([0,T'])}\leq \|f_{T'}\|_{\dot{Z}^0}
  \ls \eps$.
\end{proof}
\section{Bilinear estimates}\label{sect:bil_est}
Let $T\in (0,\infty]$. In the following, we will give estimates on the
Duhamel term
\begin{equation}\label{eq:duhamel}
  I_T(u_1,u_2)(t):=\int_0^t \mathds{1}_{[0,T)} e^{(t-t')S} \partial_x(u_1 u_2)(t') dt',
\end{equation}
which is initially defined on $C(\R;H^{1,1}(\R^2))$, and the estimates
will eventually allow us to extend this bilinear operator to larger
function spaces.
\subsection{The homogeneous case}
We start with an estimate on dyadic pieces. For a dyadic number $N$
let $A_{N}:=\{(\tau,\xi,\eta) \mid \frac12 N \leq |\xi|\leq 2N \}$.
\begin{prop}\label{prop:bil_est_dyadic}
  There exists $C>0$, such that for all $T>0$ and
  $u_{N_1},v_{N_2},w_{N_3}\in V^2_{-,S}$ such that $\supp
  \widehat{u_{N_1}}\subset A_{N_1}$, $\supp \widehat{v_{N_2}}\subset
  A_{N_2}$, $\supp \widehat{w_{N_2}}\subset A_{N_3}$ for dyadic
  numbers $N_1, N_2, N_3$ the following holds true:

  If $N_2\sim N_3$, then
  \begin{equation}\label{eq:bil_est_dyadic_a}
    \begin{split}
      &\left|\sum_{N_1\ls N_2} \int_0^T \int_{\R^2} u_{N_1} v_{N_2} w_{N_3} dx dy dt \right|\\
      \leq &C\left(\sum_{N_1\ls
          N_2}N_1^{-1}\|u_{N_1}\|^2_{V_S^2}\right)^{\frac12}
      N_2^{-\frac12}\|v_{N_2}\|_{V_S^2}N_3^{-\frac12}\|w_{N_3}\|_{V^2_S},
    \end{split}
  \end{equation}
  and if $N_1\sim N_2$, then
  \begin{equation}\label{eq:bil_est_dyadic_b}
    \begin{split}
      &\left( \sum_{N_3\ls N_2} N_3 \sup_{\|w_{N_3}\|_{V^2_S}=1}
        \left|\int_0^T \int_{\R^2} u_{N_1} v_{N_2} w_{N_3} dx dy dt
        \right|^2\right)^{\frac12}\\
      \leq &CN_1^{-\frac12}\|u_{N_1}\|_{V_S^2}
      N_2^{-\frac12}\|v_{N_2}\|_{V_S^2}.
    \end{split}
  \end{equation}
\end{prop}

\begin{proof}
  We define
  $\widetilde{u}_{N_1}=\mathds{1}_{[0,T)}u_{N_1},\widetilde{v}_{N_2}=
  \mathds{1}_{[0,T)}v_{N_2}, \widetilde{w}_{N_3}=
  \mathds{1}_{[0,T)}w_{N_3}$.  Then, we decompose
  $Id=Q_{<M} +Q_{\geq M}$, where $M$ will be chosen later, and
  we divide the integrals on the left hand side of
  \eqref{eq:bil_est_dyadic_a} into eight pieces of the form
  \begin{equation*}
    \int_{\R^3} Q^S_1\widetilde{u}_{N_1} Q^S_2\widetilde{v}_{N_2} 
    Q^S_3\widetilde{w}_{N_3} dx dy dt
  \end{equation*}
  with $Q^S_i\in\{Q^S_{\geq M},Q^S_{< M}\}$, $i=1,2,3$.  These are
  well-defined because of the $L^4$ Strichartz estimate
  \eqref{eq:strichartz_v} and \eqref{eq:mod_est3}.
  
  Let us first consider the case $Q^S_i=Q^S_{< M}$ for $1\leq i\leq
  3$. By using Plancherel's Theorem we see
  \begin{align} \label{eq:low_mod} \int_{\R^3} &
    Q^S_{<M}\widetilde{u}_{N_1} Q^S_{<M}\widetilde{v}_{N_2}
    Q^S_{<M}\widetilde{w}_{N_3} dx dy dt\\ \notag & =
    c(\widehat{Q^S_{<M}\widetilde{u}_{N_1}}*\widehat{Q^S_{<M}\widetilde{v}_{N_2}}
    *\widehat{Q^S_{<M}\widetilde{w}_{N_3}})(0)
  \end{align}
  Now, if we let $\mu_i=(\tau_i,\xi_i,\eta_i)$, $i=1,2,3$, be the
  Fourier variables corresponding to
  $\widehat{Q^S_{<M}\widetilde{u}_{N_1}}$,
  $\widehat{Q^S_{<M}\widetilde{v}_{N_2}}$, and
  $\widehat{Q^S_{<M}\widetilde{w}_{N_3}}$ respectively, then the only
  frequencies which contribute to \eqref{eq:low_mod} are those for
  which we have $\mu_1+\mu_2+\mu_3=0$.  For
  $\lambda_i=\tau_i-\xi_i^3+\frac{\eta_i^2}{\xi_i}$, $i=1,2,3$, we
  have that $|\lambda_i|<M$ within the domain of integration because of the cut off operator
  $Q^S_{<M}$.  We also have $|\xi_i|\geq N_i/2$ due to the cut off
  operators $P_{N_i}$.  By the well-known resonance identity
  \begin{equation}\label{eq:res}
    \lambda_1+\lambda_2+\lambda_3
    =3\xi_1\xi_2\xi_3+\frac{(\xi_2\eta_1-\eta_2\xi_1)^2}{\xi_1\xi_2\xi_3},
  \end{equation}
  we get
  \begin{equation}\label{eq:res1}
    \frac18 N_1 N_2 N_3 \leq |\xi_1||\xi_2||\xi_3| \leq
    \max(|\lambda_1|,|\lambda_2|,|\lambda_3|) < M 
  \end{equation}
  within the domain of integration.
  Therefore, if we set $M=8^{-1}N_1N_2N_3$ (our notation suppresses the dependence
  on $N_1,N_2,N_3$),
  it follows that
  \begin{equation*}
    \int_{\R^3} Q^S_{<M}\widetilde{u}_{N_1} Q^S_{<M}\widetilde{v}_{N_2} 
    Q^S_{<M}\widetilde{w}_{N_3} dx dy dt = 0.
  \end{equation*}

  So, let us now consider the case that $Q^S_i=Q^S_{\geq M}$ for some
  $1\leq i\leq 3$ and start with the case $i=1$.  Using the $L^4$
  Strichartz estimate \eqref{eq:strichartz_v} we obtain for
  $Q^S_2,Q^S_3\in\{Q^S_{\geq M},Q^S_{< M}\}$
  \begin{align}
    &\left|\sum_{N_1\ls N_2} \int_{\R^3} Q^S_{\geq
        M}\widetilde{u}_{N_1} Q^S_2\widetilde{v}_{N_2}
      Q^S_3\widetilde{w}_{N_3}
      dx dy dt \right|\nonumber\\
    \leq & \left\|\sum_{N_1\ls N_2} Q^S_{\geq M}\widetilde{u}_{N_1}
    \right\|_{L^2(\R^3)}\| Q^S_2\widetilde{v}_{N_2}\|_{L^4(\R^3)}
    \|  Q^S_3\widetilde{w}_{N_3}\|_{L^4(\R^3)}\label{eq:cauchy_a}\\
    \leq &C \left(\sum_{N_1\ls N_2}
      \frac{1}{N_1N_2N_3}\|\widetilde{u}_{N_1}\|^2_{V_S^2}\right)^{\frac12}
    \| Q^S_2\widetilde{v}_{N_2}\|_{V^2_{S}}\|
    Q^S_3\widetilde{w}_{N_3}\|_{V^2_{S}},\nonumber
  \end{align}
  where we used the $L^2$-orthogonality and \eqref{eq:mod_est2} on the
  first factor.  Now, we exploit \eqref{eq:mod_est3} and
  \begin{equation*}
    \|\mathds{1}_{[0,T)} f\|_{V^2_S}\leq 2 \|f\|_{V^2_S} \; ,\quad f\in V^2_{S}
  \end{equation*}
  and the claim is proved.

  We turn to the case $i=2$.  Using the interpolated bilinear
  Strichartz estimate \eqref{eq:bil_strichartz2} and
  Corollary~\ref{cor:mod_est}, we find for $Q^S_1,Q^S_3\in\{Q^S_{\geq
    M},Q^S_{< M}\}$
  \begin{align*}
    &\left|\int_{\R^3} Q^S_{1}\widetilde{u}_{N_1} Q^S_{\geq M}
      \widetilde{v}_{N_2} Q^S_3\widetilde{w}_{N_3}
      dx dy dt \right|\\
    \leq & \|Q^S_{\geq M}\widetilde{v}_{N_2} \|_{L^2(\R^3)}
    \left(\frac{N_1}{N_3}\right)^{\frac14}\|
    Q^S_1\widetilde{u}_{N_1}\|_{V^2_{S}}
    \|  Q^S_3\widetilde{w}_{N_3}\|_{V^2_{S}}\\
    \leq &\frac{C}{(N_1N_2N_3)^{\frac12}}\|v_{N_2} \|_{V^2_{S}}
    \left(\frac{N_1}{N_3}\right)^{\frac14}\|u_{N_1}\|_{V^2_{S}}
    \|w_{N_3}\|_{V^2_{S}}
  \end{align*}
  which is easily summed up with respect to $N_1\ls N_2$, because
  $N_2\sim N_3$.
  
  Finally, the case $i=3$ is proved in exactly the same way as $i=2$
  and the proof of \eqref{eq:bil_est_dyadic_a} is complete.

  In order to prove \eqref{eq:bil_est_dyadic_b}, we use the same
  decomposition as above. The case $i=1,2$, i.e.\ if the modulation on
  the first or second factor is high, we use the bilinear Strichartz
  estimate \eqref{eq:bil_strichartz2} and the claim follows similar to
  the case $i=2,3$ above.  It remains to consider the case $i=3$,
  where the modulation on the third factor is high.  Let $P_{N_3}$ be
  the projection operator onto the set $A_{N_3}$, which is symmetric.
  Therefore, using $L^2$-orthogonality and \eqref{eq:mod_est3} we
  obtain
  \begin{align}
    &\left( \sum_{N_3\ls N_2} N_3 \sup_{\|w_{N_3}\|_{V^2_S}=1}
      \left|\int_{\R^3} Q^S_{1}\widetilde{u}_{N_1} Q^S_{2}
        \widetilde{v}_{N_2} Q^S_{\geq M}\widetilde{w}_{N_3}
        dx dy dt \right|^2\right)^{\frac12}\nonumber\\
    \leq &\left( \sum_{N_3\ls N_2} \left\|P_{A_{N_3}}(
        Q^S_{1}\widetilde{u}_{N_1} Q^S_{2}
        \widetilde{v}_{N_2})\right\|^2_{L^2}
      \sup_{\|w_{N_3}\|_{V^2_S}=1} N_3 \|Q^S_{\geq
        M}\widetilde{w}_{N_3}\|^2_{L^2}
    \right)^{\frac12}\nonumber\\
    \ls & (N_1N_2)^{-\frac12} \left\|Q^S_{1}\widetilde{u}_{N_1}
      Q^S_{2} \widetilde{v}_{N_2}\right\|_{L^2}.
    \label{eq:cauchy_b}
  \end{align}
  The claim now follows from the standard $L^4$ Strichartz estimate
  \eqref{eq:strichartz} and Corollary~\ref{cor:mod_est}.
\end{proof}

\begin{theorem}\label{thm:bil_est_hom1}
  There exists $C>0$, such that for all $0<T<\infty$ and for all
  $u_1,u_2\in \dot Z^{-\frac{1}{2}}\cap C(\R;H^{1,1}(\R^2))$ it holds
  \begin{equation}\label{eq:bil_est_hom1}
    \left\|I_T(u_1,u_2)\right\|_{\dot Z^{-\frac{1}{2}}}
    \leq C \prod_{j=1}^2\|u_j\|_{\dot Y^{-\frac{1}{2}}} \;,
  \end{equation}
  and $I_T$ continuously extends to a bilinear operator
  \begin{equation*}
    I_T:\dot Y^{-\frac{1}{2}}\times \dot Y^{-\frac{1}{2}}\to \dot Z^{-\frac{1}{2}}.
  \end{equation*}
\end{theorem}
\begin{proof}
  Let $u_{1,N_1}:=P_{N_1} u_1$, $u_{2,N_2}:=P_{N_2} u_2$.  By
  symmetry, it is enough to consider the two terms
  \begin{align*}
    J_1 :=& \left\|\sum_{N_2} \sum_{N_1\ll N_2}
      I_T(u_{1,N_1},u_{2,N_2})\right\|_{\dot Z^{-\frac{1}{2}}}\\
    J_2:=&\left\|\sum_{N_2} \sum_{N_1\sim N_2}
      I_T(u_{1,N_1},u_{2,N_2})\right\|_{\dot Z^{-\frac{1}{2}}}\\
  \end{align*}
  We start with $J_1$ and fix $N$.  We may assume $N\sim N_2$ and by
  Theorem~\ref{thm:duality} and Proposition~\ref{prop:c1}
  \begin{align*}
    &
    \left\|e^{-\cdot S}P_N\sum_{N_1\ll N_2} I_T(u_{1,N_1},u_{2,N_2}) \right\|_{U^2} \\
    =& \sup_{\|v\|_{V^2}=1} \left| \sum_{N_1\ll N_2} B\left(e^{-\cdot
          S}P_N I_T(u_{1,N_1},u_{2,N_2}),v\right)
    \right|\\
    = & \sup_{\|v\|_{V^2_S}=1} \left|\sum_{N_1\ll
        N_2}\int_0^T\int_{\R^2} u_{1,N_1} u_{2,N_2} \partial_x P_N v
      dxdy dt \right|.
  \end{align*}
  We apply \eqref{eq:bil_est_dyadic_a} and obtain
  \begin{equation*}
    N^{-\frac{1}{2}}\left\|\sum_{N_1\ll N_2} P_N I_T( u_{1,N_1},u_{2,N_2})
    \right\|_{U^2_S}
    \ls
    \left(\sum_{N_1} N_1^{-1}\|u_{1,N_1}\|^2_{V_S^2}\right)^{\frac12}
    N_2^{-\frac{1}{2}}\|u_{2,N_2}\|_{V_S^2}.
  \end{equation*}
  We easily sum up the squares with respect to $N_2\sim N$.

  Next, we turn to $J_2$ and fix $N_2$.  We may assume $N\ls N_2$ and
  by Theorem~\ref{thm:duality} and Proposition~\ref{prop:c1}
  \begin{align*}
    &\sum_{N\ls N_2} N^{-1}
    \left\|e^{-\cdot S}P_N I_T(u_{1,N_1},u_{2,N_2}) \right\|^2_{U^2} \\
    =& \sum_{N\ls N_2} N^{-1}\sup_{\|v\|_{V^2}=1} \left|
      B\left(e^{-\cdot S}P_N I_T(u_{1,N_1},u_{2,N_2}),v\right)
    \right|^2\\
    = & \sum_{N\ls N_2} N^{-1} \sup_{\|v\|_{V^2_S}=1}
    \left|\int_0^T\int_{\R^2} u_{1,N_1} u_{2,N_2} P_N \partial_x v
      dxdy dt \right|^2.
  \end{align*}
  We apply \eqref{eq:bil_est_dyadic_b} and obtain
  \begin{align*}
    &\left\| \sum_{N_2}\sum_{N_1\sim N_2} I_T(u_{1,N_1},u_{2,N_2})
    \right\|_{\dot Z^{-\frac12}}
    \ls \sum_{N_2}\sum_{N_1\sim N_2} \left\| I_T(u_{1,N_1},u_{2,N_2})
    \right\|_{\dot Z^{-\frac12}}\\
    \ls &\sum_{N_2}\sum_{N_1\sim N_2} (N_1N_2)^{-\frac12} \|u_{1,N_1}\|_{V^2_S}\|u_{2,N_2}\|_{V^2_S}
  \end{align*}
   and the proof is complete.
\end{proof}

\begin{corollary}\label{cor:bil_est_hom} 
  There exists $C>0$, such that for all $0<T<\infty$ and for all
  $u_1,u_2\in \dot Z^{-\frac{1}{2}}\cap C(\R;H^{1,1}(\R^2))$ it holds
  \begin{equation}\label{eq:bil_est_hom}
    \left\|I_T(u_1,u_2)\right\|_{\dot Z^{-\frac{1}{2}}}
    \leq C \prod_{j=1}^2\|u_j\|_{\dot Z^{-\frac{1}{2}}} \;,
  \end{equation}
  and $I_T$ continuously extends to a bilinear operator
  \begin{equation*}
    I_T:\dot Z^{-\frac{1}{2}}\times \dot Z^{-\frac{1}{2}}\to \dot Z^{-\frac{1}{2}}.
  \end{equation*}
  A similar statement holds true with $\dot Z^s$ replaced by $\dot
  Y^s$.
\end{corollary}

\begin{proof}
  This is due to the continuous embedding $\dot Z^s\subset \dot Y^s$
  and Theorem~\ref{thm:bil_est_hom1}.
\end{proof}

\begin{corollary}\label{cor:scatter_limit}
  Assume that $u_1,u_2 \in \dot Y^{-\frac{1}{2}}$. Then,
  $I_\infty(u_1,u_2) \in \dot Z^{-\frac{1}{2}}$ and
  \begin{equation*}
    \|I_T(u_1,u_2)-I_\infty(u_1,u_2)\|_{\dot Z^{-\frac{1}{2}}} \to 0 \quad (T\to \infty)
  \end{equation*}
  In particular, for any $u \in \dot Y^{-\frac{1}{2}}$ it exists
  \begin{equation}\label{eq:scatter_limit}
    \lim_{t\to \infty} e^{-tS} I_\infty(u,u)(t) \in \dot H^{-\frac12,0}(\R^2).
  \end{equation}
\end{corollary}
\begin{proof}
  Without loss of generality we may assume $u_1,u_2 \in
  C(\R;H^{1,1}(\R^2))$ such that $\|u_1\|_{\dot
    Y^{-\frac12}}=\|u_2\|_{\dot Y^{-\frac12}}=1$.  Estimate
  \eqref{eq:bil_est_hom1} implies
  \begin{equation*}
    \sum_N N^{-1}\|e^{-\cdot S} P_N I_\infty(u_1,u_2)\|_{V^2_0}^2 \leq C,
  \end{equation*}
  and due to Proposition~\ref{prop:v}, Part~\ref{it:v_limits}, for all
  the dyadic pieces the limits at $\infty$ exist and we have $P_N
  I_\infty(u_1,u_2) \in V^2_{-,rc,S}$ along with
  \begin{equation*}
    \sum_N N^{-1} \|P_N I_\infty(u_1,u_2)\|_{V^2_S}^2 \leq C,
  \end{equation*}
  which yields $I_\infty(u_1,u_2)\in \dot Y^{-\frac12}$ and in
  particular the convergence \eqref{eq:scatter_limit}.

  The limits $e^{-tS}u_i(t)\to \phi_i \in \dot H^{-\frac12,0}(\R^2)$
  for $t \to \infty$ exist.  Let $\alpha_T:\R\to \R$ be
  \begin{equation} \label{eq:alpha} \alpha_T(t) = \left\{
      \begin{array}{ll}
        0 & (t<T-1)\\
        t+1-T & (T-1\leq t<T)\\
        1 & (t\geq T)
      \end{array}
    \right. 
  \end{equation}
  We define $\widetilde{u}_i=u_i-\alpha_0 e^{\cdot S}\phi_i$, $i=1,2$.
  Let $\eps>0$.  There exists $T>0$, such that $\|\alpha_T
  \widetilde{u}_i\|_{\dot Y^{-\frac12}}<\eps$, which follows by a
  similar argument as in the proof of Proposition~\ref{prop:cont_no},
  Part~\ref{it:cont_no1}. Let $T_2>T_1>T$. Then,
  \begin{equation*}
    I_{T_1}(\widetilde{u}_1,u_2)-I_{T_2}(\widetilde{u}_1,u_2)=
    I_{T_1}(\alpha_{T_1}\widetilde{u}_1,u_2)-I_{T_2}(\alpha_{T_1}\widetilde{u}_1,u_2)
  \end{equation*}
  and for $i=1,2$
  \begin{equation*}
    \|I_{T_i}(\alpha_{T_1}\widetilde{u}_1,u_2)\|_{\dot{Z}^{-\frac12}}\ls \eps.
  \end{equation*}
  By a similar argument,
  \begin{equation*}
    \|I_{T_1}(\alpha_0 e^{\cdot S}\phi_1,\tilde{u}_2)
    - I_{T_2}(\alpha_0 e^{\cdot S}\phi_1,\tilde{u}_2)\|_{\dot Z^{-\frac12}} \ls \eps.
  \end{equation*}
  On the other hand, by the $L^4$ Strichartz estimate
  \eqref{eq:strichartz_v} there exists $T'>0$ such that $\|\alpha_{T'}
  e^{\cdot S}P_N \phi\|_{L^4(\R^3)}<\eps\|P_N \phi\|_{L^2}$. For
  $T_2>T_1>T'$
  \begin{align*}
    & \|I_{T_1}( \alpha_0 e^{\cdot S}\phi_1, \alpha_0 e^{\cdot
      S}\phi_2) - I_{T_2}(\alpha_0 e^{\cdot S}\phi_1,
    \alpha_0 e^{\cdot S}\phi_2)\|_{\dot Z^{-\frac12}}\\
    = &\|I_{T_1}( \alpha_{T_1} e^{\cdot S}\phi_1, \alpha_{T_1}
    e^{\cdot S}\phi_2) - I_{T_2}(\alpha_{T_1} e^{\cdot S}\phi_1,
    \alpha_{T_1} e^{\cdot S}\phi_2)\|_{\dot Z^{-\frac12}} \ls \eps
  \end{align*}
  by the same proof as of Theorem~\ref{thm:bil_est_hom1}, using again
  Proposition~\ref{prop:bil_est_dyadic} where now the factor $\eps$
  comes from \eqref{eq:cauchy_a} and \eqref{eq:cauchy_b}. Hence, the
  family $(I_{T}(u_1,u_2))_T$ satisfies a Cauchy condition in $\dot
  Z^{-\frac12}$, which is a complete space. Therefore, it converges in
  $\dot Z^{-\frac12}$ to $I_\infty(u_1,u_2)$.
\end{proof}
\subsection{The inhomogeneous case}
\begin{prop}\label{prop:bil_est_dyadic2} Let $\eps>0$. There exists
  $C>0$ such that for all $0<T\leq 1$ and $u_{N_1}\in X$, $v_{N_2}\in
  U^2_S$, $w_{N_3}\in V^2_{-,S}$ with $\supp \widehat{u_{N_1}}\subset
  A_{N_1}$, $\supp \widehat{v_{N_2}}\subset A_{N_2}$, $\supp
  \widehat{w_{N_2}}\subset A_{N_3}$ for dyadic numbers $N_1,N_2,N_3$
  where $N_1\leq 1\leq N_2$ it holds
  \begin{equation}\label{eq:bil_est_dyadic2}
    \left|\int_0^T \int_{\R^2} u_{N_1} v_{N_2} w_{N_3} dxdy dt \right|
    \leq \frac{C(TN_1)^{\frac14-\eps}}{(N_2 N_3)^{\frac{1}{2}}}
    \|u_{N_1}\|_{X}\|v_{N_2}\|_{U_S^2}\|w_{N_3}\|_{V^2_S}\,.
  \end{equation}
\end{prop}
\begin{proof} We use the same notation as in the proof of
  Proposition~\ref{prop:bil_est_dyadic} and again the left hand side
  is well-defined. In particular we denote the time truncation of a
  function $u$ by $\widetilde{u}$.  Note that obviously
  \begin{equation*}
    \|\mathds{1}_{[0,T)} u\|_{U^2_S}\leq \|u\|_{U^2_S} \; ,\quad u\in U^2_{S}.
  \end{equation*}
  In any case we may assume that $N_3\ls N_2$, because otherwise the
  left hand side vanishes. In the first case we assume $N_1N_3^2\leq
  T^{-1}$.  Using the bilinear Strichartz estimate
  \eqref{eq:bil_strichartz1}, we obtain
  \begin{align*}
    &\left|\int_{\R^3} u_{N_1} \widetilde{v}_{N_2} \widetilde{w}_{N_3}
      dxdy dt \right|\\
    \ls & \|u_{N_1} \widetilde{v}_{N_2}\|_{L^2(\R^3)}
    \|\widetilde{w}_{N_3}\|_{L^2(\R^3)}\\
    \ls &
    T^{\frac12}\left(\frac{N_1}{N_2}\right)^{\frac{1}{2}}\|u_{N_1}\|_{U_S^2}
    \|v_{N_2}\|_{U_S^2} \|w_{N_3}\|_{V^2_S}.
  \end{align*}
  and the claim follows from $\|u_{N_1}\|_{U_S^2}\leq \|u_{N_1}\|_X$
  and $N_1^{\frac14} \leq T^{-\frac14} N_3^{-\frac12}$.

  Now, assume that $N_1N_3^2\geq T^{-1}$ and we fix $M=8^{-1}
  N_1N_2N_3$.  Recall from the proof of
  Proposition~\ref{prop:bil_est_dyadic} that we have
  \begin{equation*}
    \int_{\R^3} Q^S_{<M}\widetilde{u}_{N_1} Q^S_{<M}\widetilde{v}_{N_2} 
    Q^S_{<M}\widetilde{w}_{N_3} dx dy dt = 0.
  \end{equation*}
  Therefore we can always assume to have high modulation on one of the
  three factors.

  If $Q^S_2,Q^S_3\in\{Q^S_{\geq M},Q^S_{< M}\}$ and the modulation on
  the first factor is high, we apply the bilinear estimate
  \eqref{eq:bil_strichartz1} and Corollary~\ref{cor:mod_est} and
  obtain
  \begin{align*}
    &\left|\int_{\R^3} Q^S_{\geq M} u_{N_1} Q^S_2\widetilde{v}_{N_2}
      Q^S_3\widetilde{w}_{N_3}
      dxdy dt \right|\\
    \ls & \left(\frac{N_1}{N_2}\right)^{\frac{1}{2}}\|Q^S_{\geq M}
    u_{N_1}\|_{U_S^2} \|v_{N_2}\|_{U_S^2}\|Q^S_3
    \widetilde{w}_{N_3}\|_{L^2}
  \end{align*}
  Now, we combine this with $\|Q^S_3 \widetilde{w}_{N_3}\|_{L^2}\leq
  T^{\frac12}\|w_{N_3}\|_{V^2_S}$ and
  \begin{equation*}
    \|Q^S_{\geq M} u_{N_1}\|_{U_S^2}\ls \|Q^S_{\geq M} u_{N_1}\|_{\hx 0 {\frac12} 1}
    \ls (N_1N_2N_3)^{-\frac12}\|u_{N_1}\|_{\hx 0 1 1}
  \end{equation*}
  and \eqref{eq:bil_est_dyadic2} follows, because $N_2^{\frac12}\gs
  N_3^{\frac12}\geq T^{-\frac14}N_1^{-\frac14}$.

  If $Q^S_1,Q^S_3\in\{Q^S_{\geq M},Q^S_{< M}\}$ and the modulation on
  the second factor is high, an application of the interpolated
  estimate \eqref{eq:bil_strichartz2} yields
  \begin{align*}
    &\left|\int_{\R^3} Q^S_1 u_{N_1} Q^S_{\geq M}\widetilde{v}_{N_2}
      Q^S_3\widetilde{w}_{N_3}
      dxdy dt \right|\\
    \ls &
    \left(\frac{N_1}{N_3}\right)^{\frac{1-\eps}{2}}\|u_{N_1}\|_{V_S^2}
    \|w_{N_3}\|_{V^2_S}\|Q^S_{\geq M}\widetilde{v}_{N_2}\|_{L^2(\R^3)}\\
    \ls
    &\frac{1}{N_1^{\frac{\eps}{2}}N_2^{\frac12}N_3^{1-\frac{\eps}{2}}}\|u_{N_1}\|_{V_S^2}
    \|w_{N_3}\|_{V^2_S}\|v_{N_2}\|_{V^2_S}
  \end{align*}
  which shows the claim in this case, because
  $N_3^{\frac{1-\eps}{2}}\geq (TN_1)^{-\frac14+\frac{\eps}{2}}$.

  Finally, if $Q^S_1,Q^S_2\in\{Q^S_{\geq M},Q^S_{< M}\}$ and the
  modulation on the third factor is high, we invoke estimate
  \eqref{eq:bil_strichartz1} and obtain
  \begin{align*}
    &\left|\int_{\R^3} Q^S_1 u_{N_1} Q^S_2\widetilde{v}_{N_2}
      Q^S_{\geq M} \widetilde{w}_{N_3}
      dxdy dt \right|\\
    \ls & \left(\frac{N_1}{N_2}\right)^{\frac{1}{2}}\|Q^S_1
    u_{N_1}\|_{U_S^2} \|Q^S_1\widetilde{v}_{N_2}\|_{U^2_S} \|Q^S_{\geq
      M}
    \widetilde{w}_{N_3}\|_{L^2(\R^3)}\\
    \ls & \frac{1}{N_2N_3^{\frac{1}{2}}}\|u_{N_1}\|_{U_S^2}
    \|v_{N_2}\|_{U^2_S} \|w_{N_3}\|_{V^2_S}
  \end{align*}
  which completes the proof, because $N_2^{\frac12}\gs
  N_3^{\frac12}\geq (TN_1)^{-\frac14}$.
\end{proof}

\begin{theorem}\label{thm:bil_est_inhom} There exists $C>0$, such that
  for all $u_1,u_2\in Z^{-\frac{1}{2}}\cap C(\R;H^{1,1}(\R^2))$ it
  holds
  \begin{equation}\label{eq:bil_est_inhom}
    \left\|I_1(u_1,u_2)\right\|_{\dot Z^{-\frac{1}{2}}}
    \leq C \prod_{j=1}^2\|u_j\|_{Z^{-\frac{1}{2}}}\; ,
  \end{equation}
  and $I_1$ continuously extends to a bilinear operator
  \begin{equation*}
    I_1: Z^{-\frac{1}{2}}\times Z^{-\frac{1}{2}}\to \dot Z^{-\frac{1}{2}}
    \subset Z^{-\frac{1}{2}}.
  \end{equation*}
\end{theorem}
\begin{proof} We decompose $u_j=v_j+w_j$, $v_j\in \dot Z^{-\frac12}$
  and $w_j\in X$, $j=1,2$.  Due to $\|P_{\geq 1} u\|_{\dot
    Z^{-\frac12}}\ls \|P_{\geq 1} u\|_{X}$ and
  Corollary~\ref{cor:bil_est_hom}, it remains to prove
  \begin{align}
    \left\|I_1(P_{< 1}w_1,v_2)\right\|_{\dot Z^{-\frac{1}{2}}}
    \leq & C \|w_1\|_{X}\|v_2\|_{\dot Z^{-\frac12}}\, ,\label{eq:inhom1}\\
    \left\|I_1(P_{< 1}w_1,P_{< 1}w_2)\right\|_{\dot Z^{-\frac{1}{2}}}
    \leq & C \|w_1\|_{X}\|w_2\|_{X}.\label{eq:inhom2}
  \end{align}
  We start with a proof of \eqref{eq:inhom2}. By
  Theorem~\ref{thm:duality} and Proposition~\ref{prop:c1},
  \begin{align}
    &N^{-\frac12}\left\|P_NI_1(P_{< 1}w_1,P_{< 1}w_2)\right\|_{U^2_S}\nonumber\\
    \ls & N^{\frac12}\|P_{< 1}w_1P_{< 1}
    w_2\|_{L^1([0,1];L^2)}\nonumber\\
    \ls& N^{\frac12} \|P_{< 1}w_1\|_{\dot Z^0}\|P_{< 1}w_2\|_{\dot
      Z^0}\label{eq:small_freq}
  \end{align}
  due to the $L^4$ estimate \eqref{eq:strichartz}. We may sum up all
  dyadic pieces for $N\ls 1$.

  Let us turn to the proof of \eqref{eq:inhom1}. The estimate for
  $I_1(P_{<1}w_1,P_{<1}v_2)$ is already covered by
  \eqref{eq:small_freq}.  Assume $N_1\leq 1\leq N_2$. By
  Theorem~\ref{thm:duality} and Proposition~\ref{prop:c1}, we obtain
  \begin{align*}&N^{-\frac12}\left\|P_NI_1(P_{N_1}w_1,P_{N_2}v_2)\right\|_{U^2_S}\\
    =&N^{-\frac12}\sup_{\|f\|_{V^2_S=1}} \left|\int_0^1 \int_{\R^2}
      P_{N_1}w_{1} P_{N_2}v_{2} \partial_x P_N f
      dxdy dt \right|\\
    \ls & N_1^{\frac14-\eps}\|P_{N_1}w_1\|_X
    N_2^{-\frac12}\|P_{N_2}v_2\|_{U^2_S}
  \end{align*}
  where we applied \eqref{eq:bil_est_dyadic2} in the last step.  Now,
  we sum up with respect to $N_1\leq 1$. Finally, we perform the
  summation of the squared dyadic pieces with respect to $N\sim N_2$.
\end{proof}

\section{Proof of the main results}\label{sect:proof_main}
\noindent
In this section we present the proofs of the main results stated in
Section~\ref{sect:intro_main}.  We follow the general approach via the
contraction mapping principle, which is well-known.

For regular functions, the Cauchy problem \eqref{eq:kpII} on the time
interval $(0,T)$ for $0<T \leq \infty$ is
equivalent to
\begin{equation}\label{eq:int_eq}
  u(t)=e^{tS}u_0-\frac12 I_{T}(u,u)(t) \ , \ t\in (0,T)
\end{equation}
This allows for a generalization to rough functions: Whenever we refer
to a solution of \eqref{eq:kpII} on $(0,T)$, the operator
equation \eqref{eq:int_eq} is assumed to be satisfied.

\begin{proof}[Proof of Theorem~\ref{thm:gwp_hom}]
  Let $\alpha_0$ be as in \eqref{eq:alpha}.  We then have $\alpha_0
  e^{\cdot S}u_0 \in \dot Z^{-\frac12}$, which implies that $e^{\cdot
    S} u_0 \in \dot Z^{-\frac12}([0,\infty))$ for $u_0\in \dot
  H^{-\frac12,0}(\R^2)$ and
  \begin{equation*}
    \|e^{\cdot S}u_0\|_{\dot Z^{-\frac12}([0,\infty))}\leq \|u_0\|_{\dot H^{-\frac12,0}}.
  \end{equation*}
  Let
  \begin{equation*}
    \dot B_\delta :=\{u_0 \in \dot H^{-\frac12,0}(\R^2) \mid \|u_0\|_{\dot
      H^{-\frac12,0}}<\delta\}
  \end{equation*}
  for $\delta=(4C+4)^{-2}$, with the constant $C>0$ from
  \eqref{eq:bil_est_hom}.  Define
  \begin{equation*}
    D_r:=\{u \in \dot Z^{-\frac12}([0,\infty)) \mid \|u\|_{\dot Z^{-\frac12}([0,\infty))}\leq r\},
  \end{equation*}
  with $r=(4C+4)^{-1}$. Then, for $u_0\in \dot B_\delta$ and $u \in
  D_r$,
  \begin{equation*}
    \|e^{\cdot S}u_0-\frac12 I_\infty(u,u)\|_{\dot Z^{-\frac12}([0,\infty))}\leq
    \delta+Cr^2\leq r,
  \end{equation*}
  due to \eqref{eq:bil_est_hom} and Corollary~\ref{cor:scatter_limit}.
  Similarly,
  \begin{align*}
    & \left \|\frac12 I_\infty(u_1,u_1)- \frac12 I_\infty(u_2,u_2)
    \right\|_{\dot Z^{-\frac12}([0,\infty))} \\
    & \leq C(\|u_1\|_{\dot Z^{-\frac12}([0,\infty))}+\|u_2\|_{\dot
      Z^{-\frac12}([0,\infty))})
    \|u_1-u_2\|_{\dot Z^{-\frac12}([0,\infty))}\\
    & \leq \frac12 \|u_1-u_2\|_{\dot Z^{-\frac12}([0,\infty))},
  \end{align*}
  hence $\Phi : D_r\to D_r, u \mapsto e^{\cdot S}u_0-\frac12
  I_\infty(u,u)$ is a strict contraction. It therefore has a unique
  fixed point in $D_r$, which solves \eqref{eq:int_eq}.  By the
  implicit function theorem the map $F_+: \dot B_\delta\to D_r$,
  $u_0\mapsto u$ is analytic because the map $(u_0,u) \mapsto e^{\cdot
    S}u_0-\frac12 I_\infty(u,u)$ is analytic.  Due to the embedding
  $\dot Z^{-\frac12}([0,\infty))\subset C([0,\infty),\dot
  H^{-\frac12}(\R^2))$ the regularity of the initial data persists
  under the time evolution. Concerning the results
  with respect to the negative time axis, we reverse the time
  $t\mapsto -t$ and apply the same arguments.
\end{proof}

\begin{remark}\label{rem:uniqueness}
  Up to now, we only know that the solution $u$ is unique in the
  subset $D_r\subset \dot Z^{-\frac12}([0,\infty))$.  The proof of the
  uniqueness assertion in the larger space
  $Z^{-\frac12}([0,T])$ will follow from the results in the
  subsequent subsection.
\end{remark}

\begin{proof}[Proof of Corollary~\ref{cor:scatter}]
  For initial data $u_0\in \dot H^{-\frac12,0}(\R^2)$, $\|u_0\|_{\dot
    H^{-\frac12,0}}<\delta$, the solution $u$ which was constructed
  above satisfies
  \begin{equation*}
    u(t)=e^{tS}(u_0-e^{-\cdot S}\frac12 I_{\infty}(u,u))(t) \ , \ t\in (0,\infty)
  \end{equation*}
  The existence of the limit $u_0-e^{-tS}\frac12 I_{\infty}(u,u)(t)\to
  u_+$ in $\dot H^{-\frac12,0}(\R^2)$ as $t\to \infty$ follows from
  Corollary~\ref{cor:scatter_limit}. The analyticity of the map
  $V_+:u_0\mapsto u_+$ follows from the analyticity of $F_+$ shown
  above.

  An obvious modification of the above proof also yields
  persistence of higher initial regularity, in particular if $u_0\in
  L^2(\R^2)$, then $u(t)\in L^2(\R^2)$ for all $t$.
  It remains to show $\|V_+(u_0)\|_{L^2}=\|u_0\|_{L^2}$. By approximation
  and a direct calculation for smooth solutions, we easily see that
  the $L^2$-norm is conserved. Due to the strong convergence in $\dot
  H^{-\frac12,0}(\R^2)$ we have weak convergence
  $e^{-tS}u(t)\rightharpoonup u_+$ in $L^2(\R^2)$ for $t\to \infty$,
  hence $\|u_+\|_{L^2}\leq \|u_0\|_{L^2}$.
  Let $F_-$ be the flow map with respect to $(-\infty,0)$
  according to Theorem~\ref{thm:gwp_hom},
  which is Lipschitz continuous. Because $\lim_{t\to \infty}e^{tS}u_+-u(t)=0$
  in $\dot H^{-\frac12,0}(\R^2)$, it follows
  $\lim_{t\to \infty}F_-(-t,e^{tS}u_+)=u_0$ in $\dot H^{-\frac12,0}(\R^2)$.
  Moreover, due to the $L^2$ conservation
  $\|u_+\|_{L^2}=\|F_-(-t,e^{tS}u_+)\|_{L^2}$ we have weak convergence
  $F_-(-t,e^{tS}u_+)\rightharpoonup u_0$ in $L^2(\R^2)$. Altogether,
  \begin{equation*}
    \|u_0\|_{L^2}\leq \lim_{t\to \infty} \|F_-(-t,e^{tS}u_+)\|_{L^2}=\|u_+\|_{L^2}.
  \end{equation*}
  
  The existence and analyticity of the local inverse $W_+$ follows
  from the inverse function theorem, because $V_+(0)=0$ and by
  \eqref{eq:bil_est_hom} and Corollary~\ref{cor:scatter_limit} we
  observe $D V_+(0)=Id$.
\end{proof}

\begin{proof}[Proof of Theorem~\ref{thm:lwp_hom_inhom}]
  For some $\delta>0$ and $R\geq \delta$ we define
  \begin{equation*}
    B_{\delta,R}:=\{u_0 \in H^{-\frac12,0}(\R^2) \mid
    u_0=v_0+w_0, \, \|v_0\|_{\dot H^{-\frac12,0}}<\delta, \|w_0\|_{L^2}<R\}.
  \end{equation*}
  Let $u_0 \in B_{\delta,R}$ with $u_0=v_0+w_0$.  We have $\chi
  e^{\cdot S} w_0 \in X$ and $\chi e^{\cdot S} v_0\in \dot
  Z^{-\frac12,0}$, which implies that $e^{\cdot S} u_0 \in
  Z^{-\frac12}([0,1])$ and
  \begin{equation*}
    \|e^{\cdot S}u_0\|_{Z^{-\frac12}([0,1])}\ls \delta+R.
  \end{equation*}
  We start with the case $R=\delta=(4C+4)^{-2}$, with the constant
  $C>0$ from \eqref{eq:bil_est_inhom}.  Define
  \begin{equation*}
    D_r:=\{u \in Z^{-\frac12}([0,1]) \mid \|u\|_{Z^{-\frac12}([0,1])}\leq r\},
  \end{equation*}
  with $r=(4C+4)^{-1}$. As above, we use \eqref{eq:bil_est_inhom} to
  verify that
  \begin{equation*}
    \Phi : D_r\to D_r, u \mapsto e^{\cdot S}u_0-\frac12 I_1(u,u)
  \end{equation*}
  is a strict contraction, for $u_0\in \dot B_{\delta,R}$.  It
  therefore has a unique fixed point in $D_r$, which solves
  \eqref{eq:int_eq} on the interval $(0,1)$.  By the implicit function
  theorem the map $B_{\delta,R}\to D_r$, $u_0\mapsto u$ is analytic.
  We also have the embedding $Z^{-\frac12}([0,1])\subset
  C([0,1];H^{-\frac12}(\R^2))$.  Now, we assume that $u_0\in
  B_{\delta,R}$ for $R\geq \delta=(4C+4)^{-2}$. We define
  $u_{0,\lambda}=\lambda^2 u_0(\lambda \cdot,\lambda^2 \cdot)$. For
  $\lambda=R^{-2}\delta^2$ we observe $u_{0,\lambda} \in
  B_{\delta,\delta}$. Therefore we find a solution $u_\lambda \in
  Z^{-\frac12,0}([0,1])$ on $(0,1)$ with $u_\lambda(0)=u_{0,\lambda}$.
  By rescaling \eqref{eq:scaling} we find a solution $u \in
  Z^{-\frac12,0}([0,\delta^6R^{-6}])$ on $(0,\delta^6R^{-6})$ with
  $u(0)=u_0$. We notice that in \eqref{eq:bil_est_inhom}, the left
  hand side is in the homogeneous space $\dot Z^{-\frac12,0}$, hence
  all of the above remains valid (or even becomes easier) if we
  exchange $Z^{-\frac12,0}([0,1])$ by the smaller space $\dot
  Z^{-\frac12,0}([0,1])$.

  It remains to show the uniqueness claim.  Assume that $u_1,u_2 \in
  Z^{-\frac12,0}([0,T])$ are two solutions such that
  $u_1(0)=u_2(0)$.  Moreover, we assume that
  \begin{equation*}
    T':=\sup\{0\leq t\leq T \mid u_1(t)=u_2(t)\}<T.
  \end{equation*}
  By a translation in $t$ it is enough to consider $T'=0$. A
  combination of \eqref{eq:bil_est_hom} and \eqref{eq:bil_est_dyadic2}
  yields the following: Decompose $u_j=v_j+w_j$, where $v_j\in
  X([0,T])$, $w_j \in \dot Z^{-\frac12,0}([0,T])$ and $w_j(0)=0$.
  Then, there exists $C>0$, such that for all small $0<\tau \leq T'$
  \begin{align*}
    & \|u_1-u_2\|_{Z^{-\frac12,0}([0,\tau])}
    =\left\|\frac12 I_\tau(u_1,u_1)-\frac12 I_\tau(u_2,u_2)\right\|_{Z^{-\frac12,0}([0,\tau])}\\
    \leq{} & C \tau^{\frac14-\eps}
    \left(\|v_1\|_{X([0,\tau])} +\|v_2\|_{X([0,\tau])}\right)\|u_1-u_2\|_{Z^{-\frac12,0}([0,\tau])}\\
    +& C \left(\|w_1\|_{\dot Z^{-\frac12,0}([0,\tau])} +\|w_2\|_{\dot
        Z^{-\frac12,0}([0,\tau])}\right)
    \|u_1-u_2\|_{Z^{-\frac12,0}([0,\tau])}.
  \end{align*}
  We apply Proposition~\ref{prop:cont_no}, Part~\ref{it:cont_no2} and
  obtain
  \begin{equation*}
    \|u_1-u_2\|_{Z^{-\frac12,0}([0,\tau])}
    \leq  \frac12 \|u_1-u_2\|_{Z^{-\frac12,0}([0,\tau])},
  \end{equation*}
  which contradicts the definition of $T'$.
\end{proof}
\section*{Erratum}\label{sect:erratum}
In this section we will point out and correct an error which has occured in Section \ref{sect:disp_est} and which came to our attention after publication of this work. It has no implications for the main results of the paper, but some adjustments are necessary in the section on $U^p$ and $V^p$ spaces.

In Definition \ref{def:v} of the space $V^p$ we included the normalizing condition $\lim_{t\to \infty}v(t)=0$. This, however, leads the problem that Theorem \ref{thm:duality} on duality is incorrect as stated: Indeed, for fixed nonzero $\phi \in L^2$ consider the functional $T(u)=\lb \lim_{t \to \infty} u(t),\phi\rb$ on $U^p$. This cannot be represented by $V^{p'}$ functions via the bilinear form $B$. The following modifications can be performed in order to resolve this issue:

\begin{enumerate}
\item Define $\mathcal{Z}_0$ as follows: $\mathcal{Z}_0$ is defined as the set
of finite partitions $-\infty<t_0<t_1<\ldots<t_K\leq \infty$.
\item Remove the notational convention for $u(-\infty)$ and $u(\infty)$ from the Proposition \ref{prop:u}, item \ref{it:u_limits} (this change is purely notational).
\item  We define $V^p$ as the normed space of all
  functions $v:\R\to L^2$ such that $\lim_{t\to\pm
    \infty}v(t)$ exist and for which the norm
  \begin{equation*}
    \|v\|_{V^p}:=\sup_{\{t_k\}_{k=0}^K \in \mathcal{Z}}
    \left(\sum_{k=1}^{K}
      \|v(t_{k})-v(t_{k-1})\|_{L^2}^p\right)^{\frac{1}{p}}
  \end{equation*}
  is finite, where we use the convention that $v(-\infty)=\lim_{t\to -\infty}$ and $v(\infty)=0$ (Here, the difference is that $v(\infty)=0$ does not necessarily conincide with the limit at $\infty$). This convention will also be used in the sequel.

Likewise, let $V^p_-$ denote the closed subspace of all $v \in V^p$ with $\lim_{t\to -\infty}v(t)=0$ (Note that the space $V^p_-$ is unchanged).
\item Proposition \ref{prop:gen_deriv} on the bilinear form remains unchanged. However, notice that our convention is $v(t_K)=0$ since $t_K=\infty$ for all partitions $\{t_k\}_{k=0}^K \in \mathcal{Z}$.
\item In the proof of Theorem \ref{thm:duality} we have $\tilde{v}(t)=v(t)$ and the error in the calculation at the end of the proof is corrected.
\item The conclusion of Proposition \ref{prop:c1} remains valid with the modified definition of $V^p$ by a similar proof. Alternatively, it can be seen as follows: We know that it is correct for $v-\lim_{t\to \infty}v(t)$ as proved in Section \ref{sect:disp_est} and we obtain
\[
B(u,v-\lim_{t\to \infty}v(t))=-\int_{-\infty}^\infty \lb u'(s),v(s)-\lim_{t\to \infty}v(t)\rb ds
\]
which is equivalent to
\[
B(u,v)+\lim_{t\to \infty}\lb u(t),v(t)\rb =-\int_{-\infty}^\infty \lb u'(s),v(s)-\lim_{t\to \infty}v(t)\rb ds
\]
which shows the claim.
\end{enumerate}

\section*{Acknowledgments}
The research of this work has been supported by the Sonderforschungsbereich 611 (DFG).

\end{document}